\documentclass[11pt]{amsart}
\pdfoutput=1
\usepackage{graphicx}
\usepackage{pdfsync}
\usepackage{fncylab}
\usepackage{mathrsfs}
\usepackage{amsthm}
\usepackage{amsmath}
\usepackage{amssymb}
\usepackage[T1]{fontenc}
\usepackage{caption}
\usepackage{subcaption}
\usepackage[usenames,dvipsnames,svgnames,table]{xcolor}
\usepackage[normalem]{ulem}
\usepackage[multiple]{footmisc}
    \newtheoremstyle{upright}%
        {1pt plus1pt minus1pt}%
        {1pt plus1pt minus1pt}%
        {\upshape}%
        {}%
        {\bfseries\scshape}%
        {\textbf{.}}%
        {1em}%
        {}%

\newtheorem{theorem}{Theorem}[section]

\newtheorem{corollary}[theorem]{Corollary}
\newtheorem{lemma}[theorem]{Lemma}

\newtheorem{proposition}[theorem]{Proposition}

\theoremstyle{definition}   
\newtheorem{definition}[theorem]{Definition}

\newtheorem{example}[theorem]{Example}
\newtheorem{notation}[theorem]{Notation}
\newtheorem{remark}[theorem]{Remark}

\newcommand{\R}{\mathbb{R}} 
\newcommand{\Z}{\mathbb{Z}} 

\newcommand{\xoo}{x_0^0} 
\newcommand{\xio}[1]{x_{#1}^0} 
\newcommand{\xii}[2]{x_{#1}^{#2}} 

\newcommand{\billiardOrbit}{\mathscr{O}(x^0,\theta^0)} 

\newcommand{\billiardOrbitang}[1]{\mathscr{O}(x^0,#1)} 

\newcommand{\orbiti}[1]{\mathscr{O}_{#1}(\xio{#1},\theta_{#1}^0)} 
\newcommand{\orbitix}[2]{\mathscr{O}_{#1}(#2,\theta_{#1}^0)} 
\newcommand{\orbitiang}[2]{\mathscr{O}_{#1}(\xio{#1},#2)} 
\newcommand{\orbitixang}[3]{\mathscr{O}_{#1}(#2,#3)} 

\newcommand{\seqang}[1]{\{\mathscr{O}_n(\xio{n},#1)\}_{n=0}^\infty}
\newcommand{\seqi}[1]{\{\mathscr{O}_n(\xio{n},\theta^0)\}_{n=#1}^\infty}
\newcommand{\seqii}[2]{\{\mathscr{O}_{#2}(\xio{#2},\theta^0)\}_{#2=#1}^\infty}
\newcommand{\seqixang}[3]{\{\mathscr{O}_n(#2,#3)\}_{n=#1}^\infty}
\newcommand{\seqiang}[2]{\{\mathscr{O}_n(\xio{n},#2)\}_{n=#1}^\infty}

\newcommand{\ntpi}[1]{\mathscr{N}_{#1}(\xio{#1},\theta^0)}
\newcommand{\ntpixang}[3]{\mathscr{N}_{#1}(#2,#3)}
\newcommand{\ntp}{\mathscr{N}(x^0,\theta^0)}

\newcommand{\ntpang}[1]{\mathscr{N}(x^0,#1)}

\newcommand{\tfractal}{\mathscr{T}}

\newcommand{\tfraci}[1]{\mathscr{T}_{#1}}
\newcommand{\omegat}{\Omega(\mathscr{T})}
\newcommand{\omegati}[1]{\Omega(\mathscr{T}_{#1})}

\begin{document}

\title[Nontrivial paths and periodic orbits]{Nontrivial paths and periodic orbits of the $T$-fractal billiard table.}

\author[M. L. Lapidus]{Michel L. Lapidus}
\address{University of California, Department of Mathematics, 900 Big Springs Rd., Riverside, CA  92521-0135, USA}
\email{lapidus@math.ucr.edu}
\thanks{The work of M. L. Lapidus was partially supported by the National Science Foundation under the research grants DMS-0707524 and DMS-1107750, as well as by the Institut des Hautes Etudes Scientifiques (IHES) in Bures-sur-Yvette, France, where he was a visiting professor while part of this paper was written. The work of R. G. Niemeyer was partially supported by the National Science Foundation under the MCTP grant DMS-1148801, while a postdoctoral fellow at the University of New Mexico, Albuquerque.}
\author[R. L. Miller]{Robyn L. Miller}
\address{The Mind Research Network, Albuquerque, NM 87106, USA}
\email{rmiller@mrn.org}

\author[R. G. Niemeyer]{Robert G. Niemeyer}
\address{University of New Mexico, Department of Mathematics and Statistics, 311 Terrace NE, Albuquerque, NM  87131-0001, USA}
\email{niemeyer@math.unm.edu}


\begin{abstract}
We introduce and prove numerous new results about the orbits of the $T$-fractal billiard.  Specifically, in \S\ref{sec:SequencesOfCompatiblePeriodicOrbits}, we give a variety of sufficient conditions for the existence of a sequence of compatible periodic orbits.   In \S\ref{sec:NontrivialPathsInTheTFractalBilliardTable}, we examine the limiting behavior of particular sequences of compatible periodic orbits. Additionally, sufficient conditions for the existence of particular nontrivial paths are given in \S\ref{sec:NontrivialPathsInTheTFractalBilliardTable}. The proofs of two results stated in \cite{LapNie4} appear here for the first time, as well.  In \S\ref{sec:aNontrivialPathInAnIrrationalDirection}, an orbit with an irrational initial direction reaches an elusive point in a way that yields a nontrivial path of finite length, yet, by our convention, constitutes a singular orbit of the fractal billiard table.    The existence of such an orbit seems to indicate that the classification of orbits may not be so straightforward.  A discussion of our results and directions for future research is then given in \S\ref{sec:Discussion}.
\end{abstract}

\maketitle
\setcounter{tocdepth}{2}
\tableofcontents

\section{Introduction}
\label{sec:introduction}
A fractal billiard table is a planar billiard table $\Omega(F)$ where the boundary $\partial\Omega(F)=F$ is a fractal curve.  In this paper, we take $\Omega(F)$ to be the $T$-fractal billiard table shown in Figure \ref{fig:tfractal}.  In \cite{CheNie1}, the fractal billiard table is a self-similar Sierpinski carpet billiard table.  In \cite{LapNie1,LapNie2,LapNie3}, the fractal billiard table under consideration was the Koch snowflake fractal billiard table.  In \cite{LapNie4}, recent results on the Koch snowflake fractal billiard table and a self-similar Sierpinski fractal billiard table are surveyed and the $T$-fractal billiard table is introduced; preliminary results regarding the $T$-fractal billiard table were presented without proof.  The $T$-fractal is not, strictly speaking, a self-similar set nor is it the finite union of self-similar sets.  However, as we demonstrate at the end of this introduction, $\omegat$ can be constructed by way of a particular iterated function system whose contraction mapping is denoted by $\Phi$.  We illustrate in Figure \ref{fig:tfractal} how to construct the $T$-fractal billiard.

We mention that the $T$-fractal is \textit{not} the unique fixed point attractor of $\Phi$, this being a technical detail that does not take away from the fractality of the set.  Rather, the unique fixed point attractor of $\Phi$ is of great interest and is hereafter referred to as the set of elusive points of the fractal billiard table; such points are those that are never found in any finite approximation and a formal definition of \textit{elusive point} is given in Definition \ref{def:ElusivePoint}, with further elaboration given in Remark \ref{rmk:elusivePointsNotAlwaysAttractor}.  The geometry of $\omegat$ will aide us in our analysis of nontrivial paths (in the sense of Definition \ref{def:ANontrivialPath}) and periodic orbits of the $T$-fractal billiard table.

\begin{figure}
\begin{center}
\includegraphics[scale=.65]{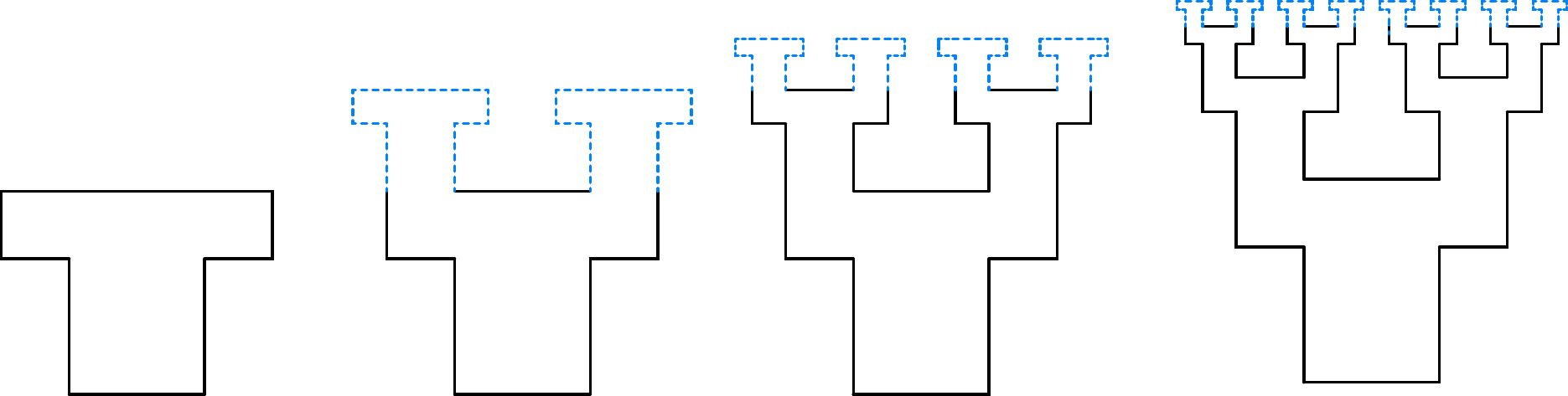}
\end{center}
\caption{The construction of the T-fractal billiard table.  Note that, at each level $n$, $n\geq 0$, $2^{n+1}$ copies of the base $T$ (shown on the left) are appended, as shown in blue, (or, for those not viewing the color version of this article, as dotted segments) so as to construct the $n+1$ approximation.}
\label{fig:tfractal}
\end{figure}

The $T$-fractal enjoys particular properties not found in the Koch snowflake and  a self-similar Sierpinski carpet.  Indeed, the Koch snowflake $K\!S$ is a nowhere differentiable curve bounding a region with finite area.  Furthermore, a self-similar Sierpinski carpet contains no area.  By contrast, the T-fractal bounds a region with finite area, yet a nontrivial portion of the boundary remains differentiable.  Hence, determining directions for which one gets periodic orbits is a much easier task compared to finding such orbits in the Koch snowflake fractal billiard table.  Additionally, the set of directions for which one finds periodic orbits of the $T$-fractal billiard table is countably infinite, this being in contrast to a self-similar Sierpinski carpet billiard table, where the number of directions for which periodic orbits occur is finite (this finite value depending exclusively on the scaling ratio of the carpet).  We note that Propositions \ref{thm:WhatTheSlopeCannotBeInTFrac} and \ref{thm:WhatTheSlopeCanBeInTFrac} below first appeared without proof and were stated under more restrictive conditions in \cite{LapNie4}; in this paper, we provide detailed proofs of these more general propositions, among many other results.

The overarching themes of the research program on fractal billiards and fractal flat surfaces are topological and measure-theoretic in nature.  In one instance, we seek to understand what constitutes a periodic orbit of a fractal billiard table.  In another instance, we seek to understand when an orbit may be dense in a fractal billiard table.  More generally, we eventually seek to answer the question asking which directions yield flows that are closed and which yield  flows that are ergodic (i.e., where almost every orbit is uniformly distributed in the fractal billiard table).  

Of course, all of this relies on the existence of a well-defined billiard map and phase space.  Defining such a map and space for a fractal billiard table is a very difficult task, one that the authors have yet to accomplish.  In order to work towards the abovementioned goals in the absence of a billiard map and phase space, we rely on the geometric properties of the fractal to construct what we are calling nontrivial paths of a fractal billiard table.  We are ultimately interested in nontrivial paths reaching elusive points\footnote{An elusive point is never a point of any finite approximation and is formally defined in Definition \ref{def:ElusivePoint}.} of a fractal billiard table.  In the case of the $T$-fractal billiard (and, likewise, the Koch snowflake fractal billiard table), any pointmass on a trajectory that results in the pointmass reaching an elusive point does so in an increasingly confined manner.  It is then natural to ask how a pointmass will ``reflect'' off from an elusive point and under what conditions will such an orbit

\begin{enumerate}
	\item form a periodic orbit of the fractal billiard table (i.e., return to the elusive point infinitely often in the same direction),
	\item terminate at an elusive point, because of some sort of inherent ambiguity,
	\item or manage to fill all or part of some nontrivial 2-dimensional subset of $\omegat$ (either uniformly or not).
\end{enumerate}

We answer in part these questions by giving a variety of sufficient conditions under which 1) a periodic orbit of $\omegat$ exists and, more generally, 2) a nontrivial path\footnote{A path that manages to never be confined to any finite approximation of $\omegat$.}  of $\omegat$ exists; see Definition \ref{def:ANontrivialPath}.  In the case of the $T$-fractal billiard (and, also, the Koch snowflake fractal billiard), there are trajectories yielding orbits that reach elusive points.  The geometric properties of the $T$-fractal billiard lend themselves well to the description of such trajectories, this being the major focus of \S\ref{sec:SequencesOfCompatiblePeriodicOrbits} and \S\ref{sec:NontrivialPathsInTheTFractalBilliardTable}.

The paper is organized as follows.  In \S\ref{sec:background}, the necessary background is given so that a reader not familiar with mathematical billiards or the terminology developed in earlier works may be able to understand the remainder of the paper.  Sufficient conditions for sequences of compatible periodic orbits are given in \S\ref{sec:SequencesOfCompatiblePeriodicOrbits}.  In \S\ref{sec:NontrivialPathsInTheTFractalBilliardTable}, we build upon the concepts developed in \S\ref{sec:SequencesOfCompatiblePeriodicOrbits} and show (via an explicit construction) that there exist nontrivial paths and periodic orbits in the $T$-fractal billiard $\omegat$.  Then, in \S\ref{sec:aNontrivialPathInAnIrrationalDirection}, we introduce an important example of a nontrivial path with an initial direction that is irrational, yet reaches an elusive point of the $T$-fractal billiard in a way that is similar to how particular nontrivial paths with rational initial directions reach their respective elusive points.  This is an unintuitive behavior that begins to depart from the classical theory of billiards on square tiled billiard tables.  Finally, in \S\ref{sec:Discussion}, we discuss some of our results and propose directions for future research.

We close this introduction by giving a more detailed geometric description of the $T$-fractal billiard.  As we mentioned in the preceding text, one can construct the T-fractal billiard by way of a particular iterated function system $\Phi = \{\phi_1,\phi_2\}$ given by
\begin{align}
 \phi_1(\mathbf{x}) = \frac{1}{2}\mathbf{x} + \left(1,\frac{3}{2}\right) & \quad
\phi_2(\mathbf{x}) = \frac{1}{2}\mathbf{x} + \left(\frac{-1}{2},\frac{3}{2}\right).
\label{eqn:IFS}
\end{align}
Let $\omegati{0}$ be a $1\times 1$ square unioned with a $2\times 1/2$ rectangle, as shown in Figure \ref{fig:T0}.  According to the convention which we have used in our exact arithmetic simulations, the lower left corner of the $1\times 1$ square is the origin $(0,0)$.  In this paper, the choice for the location of the origin is arbitrary.  However, in current works in progress, it is convenient, if not necessary, to require that the base of $\omegat$ be exactly the unit interval $[0,1]$.\footnote{Current works in progress with C. C. Johnson focus on developing a fractal interval exchange transformation on the $T$-fractal flat surface and investigating the behavior of a particular orbit of the $T$-fractal billiard table; see \cite{JohNie1,JohNie2}, respectively.  So as to maintain consistency across articles, we designate the origin as indicated in Figure \ref{fig:T0}.} If $\omegati{n}:= \bigcup_{j=0}^n\Phi^j(\omegati{0})$, where 
\begin{align}
\Phi^j(\cdot):&=\bigcup^2_{i_1,\cdots, i_j=1} \phi_{i_1}\circ\phi_{i_1}\circ\cdots\circ\phi_{i_j}(\cdot),
\end{align}
\noindent then we define $\omegat$ by $\overline{\bigcup_{n=0}^\infty \omegati{n}}$.  

\begin{figure}
	\begin{center}
	\includegraphics{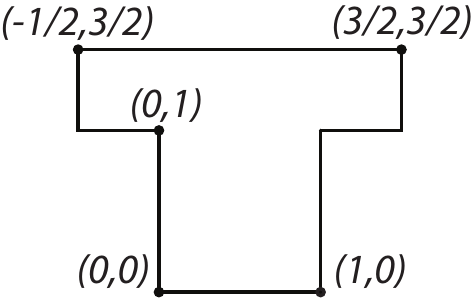}
	\end{center}
	\caption{The geometry of the  set $\omegati{0}$.}
	\label{fig:T0}  
\end{figure}

As we previously mentioned, the set of elusive points, which we denote by $\mathscr{E}$, is the unique fixed point attractor of $\Phi$.  Since $\Phi([-1,2]\times \{3\}) = [-1,2]\times \{3\}$ and the attractor of $\Phi$ is unique, we see that $\mathscr{E} = [-1,2]\times \{3\}$.
\section{Background}

\label{sec:background}
For the reader's easy reference, we provide in this section various definitions appearing in previous joint works; see \cite{CheNie1, LapNie1, LapNie2, LapNie3, LapNie4}.  However, the notation in all of the definitions will reflect the fact that we are discussing the $T$-fractal billiard table, which will significantly simplify the exposition.  Particular definitions (e.g., the definitions for \textit{compatible initial conditions} and \textit{a sequence of compatible initial conditions}) will be phrased in full generality so that the reader may be able to understand the results appearing in the abovementioned works.


In order to understand the following notations and definitions, we first discuss the phase space for the billiard dynamics on a polygonal billiard table. If $\Omega(B)$ is a polygonal billiard with boundary $B$, then we consider the Cartesian product $\Omega(B)\times S^1$, where $S^1$ is the unit	 circle in $\mathbb{R}^2$.  As a subset of the tangent bundle on $\mathbb{R}^2$, the flow on $\Omega(B)\times S^1$ is only partially defined, since any trajectory on $\Omega(B)$ would intersect with $B$.  For points of $\Omega(B)$ that are not vertices, one can define an equivalence relation $\sim$ on $\Omega(B)\times S^1$ that identifies outward pointing vectors with inward pointing vectors so that the flow on $\Omega(B)\times S^1$  can be continued in a continuous manner at the boundary.  That is, if $\rho_i$ is the reflection through the edge $e_i$ of $B$, then for a point $x$ on $e_i$ which is not a vertex of $B$, we write that $(x,\theta)\sim (x,\rho_i(\theta))$; for points in the interior of $\Omega(B)$, $(x,\theta)$ is equivalent to itself and to no other ordered pair. 

However, for points of $B$ that are vertices, a geodesic cannot always be continued in a well-defined manner.  A priori, one terminates the geodesic when it intersects a vertex of $B$.  More precisely, if $x$ is a vertex of $B$, then $(x,\theta)\sim (x,\rho(\theta))$ for every $\rho$ in the group generated by $\rho_i$ and $\rho_j$, the reflections through the edges $e_i$ and $e_j$, respectively.  As one may then see, in certain situations, one can show that a geodesic intersecting a vertex of $B$ can be continued in a well-defined manner.  This occurs when there is only one $\rho'$ in the group generated by the elements $\rho_i$ and $\rho_j$ such that $\rho'(\theta)$ is an inward pointing direction at the vertex $x$.  Thus, one can define the quotient space $(\Omega(B)\times S^1)/\!\sim$ to be the tangent bundle of $\Omega(B)$.  

 In the context of a prefractal approximation $\omegati{n}$ of the $T$-fractal billiard, this is when a billiard orbit intersects a right angle of $\omegati{n}$ (when one measures the angle from the interior of $\omegati{n}$).  When a billiard orbit intersects a vertex with an obtuse angle (still when measured from the interior of $\omegati{n}$), one cannot determine the continuation of the geodesic in a well-defined manner.  Hence, the geodesic terminates.

In general, one restricts one's attention to $(B\times S^1)/\!\sim$.  When defined, the map $f_B\!:(B\times S^1)/\!\sim \,\,\rightarrow (B\times S^1)/\!\sim$ is a map on the set of  equivalence classes.  The representative element of an equivalence class $[(x,\theta)]$ is the element $(x,\theta)$, where $\theta$ is the inward pointing direction based at $x$.  As such, one simplifies notation slightly by considering $f_B$ (again, when defined) as a map on the representative elements; i.e., $f_B(x,\theta) = (x',\rho_i(\theta'))$, where $\rho_i(\theta')$ is the reflection of the angle $\theta'$ through the side $e_i$ containing $x'$.  The map $f_B$ is called the \textit{billiard map} (of $\Omega(B)$).

\begin{notation}
An initial condition of an orbit of a billiard table $\Omega(B)$ is given by $(x^0,\theta^0)$, where $x^0$ is the initial basepoint on the boundary $B$ and $\theta^0$ is the initial inward pointing direction; see Figure \ref{fig:billiardMap}.  In the $T$-fractal billiard table approximation $\omegati{n}$, the initial condition of an orbit will be given by $(\xio{n},\theta^0_n)$.  If $f_B$ is the billiard map\footnote{See \cite{Sm} and \S2 of \cite{LapNie4} for a detailed discussion of the billiard map $f_B$ and the phase space $(B\times S^1)/\!\sim$, including the equivalence relation $\sim$.} describing the flow in the phase space $(B\times S^1)/\!\sim$, then $f_B^k(x^0,\theta^0) = (x^k,\theta^k)$, the ($k+1$)th point-angle pair in the orbit $\mathscr{O}_B(x^0,\theta^0)$.  If we are considering an approximation of the $T$-fractal billiard table, then an orbit of $\omegati{n}$ is given by $\orbiti{n}$ and $f_n^{k}(\xio{n},\theta^0_n) = (\xii{n}{k},\theta^{k}_n)$ is the ($k+1$)th point-angle pair in the orbit of $\omegati{n}$.\footnote{We stress that $k$ depends on $n$; however, it is unnecessary to indicate so in the notation.  Later, we will introduce the notion of \textit{first return time} and \textit{first escape time} of an orbit of a prefractal billiard $\omegati{n}$, and the explicit dependence on $n$ in the notation for such notions will be clearly indicated.}
\label{not:orbitNotation}
\end{notation}

\begin{figure}
\begin{center}
\includegraphics{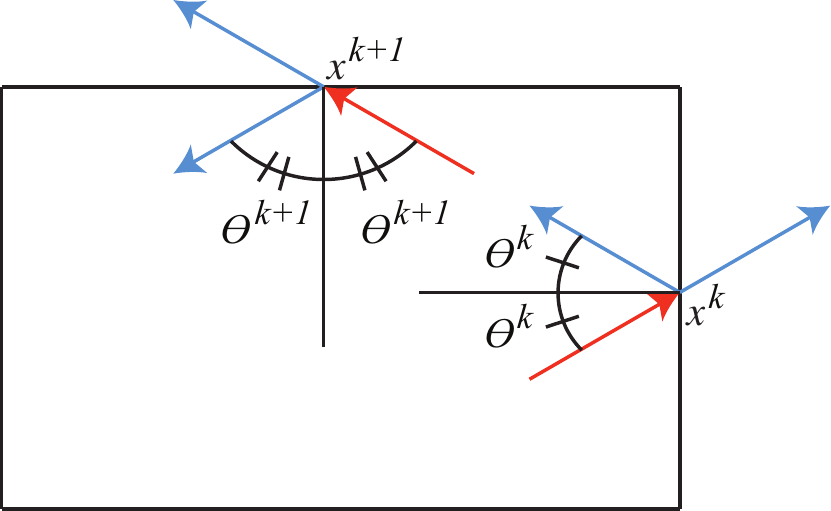}
\end{center}
\caption{One always takes the direction of motion to be the inward pointing vector, this being identified with the outward pointing vector by way of the equivalence relation $\sim$ defined in the text prior to Notation \ref{not:orbitNotation}.  Here, we see a billiard ball beginning at $x^k$ and then going in the direction of $\theta^k$.  Upon collision at the point $x^{k+1}$, the angle of reflection equals the angle of incidence and the billiard ball continues.}
\label{fig:billiardMap}
\end{figure}


\begin{definition}[Elusive point]
\label{def:ElusivePoint}
Let $\omegat$ be the $T$-fractal billiard table approximated by a sequence of rational polygonal billiard tables $\omegati{n}$ (as shown in Figure \ref{fig:tfractal}), where  $\omegati{n}\subseteq\omegat$ for every $n\geq 0$.  Then 
\begin{align}
\notag\omegat \setminus&\bigcup_{n=0}^\infty\omegati{n}	
\end{align}
\noindent is the collection of all \textit{elusive points} of $\omegat$.  We denote the set of elusive points of $\omegat$ by $\mathscr{E}$.
\end{definition}

\begin{remark}
\label{rmk:elusivePointsNotAlwaysAttractor}
As we noted in the introduction, the set of elusive points $\mathscr{E}$ of the $T$-fractal billiard is the attractor of the iterated function system $\Phi$ introduced at the end of \S\ref{sec:introduction}.  While Definition \ref{def:ElusivePoint} remains valid for the Koch snowflake fractal billiard, for example, the set of elusive points for the Koch snowflake fractal will not be closed, nor will it be the unique fixed point attractor associated with the Koch snowflake fractal billiard.  The fact that the set of elusive points $\mathscr{E}$ of the $T$-fractal billiard is a closed set in the plane is specific to the $T$-fractal billiard and not a property that we will generally observe when investigating the billiard dynamics on other fractal billiard tables.
\end{remark}

Being an interval in the plane, the set $\mathscr{E}$, though stretched, is a copy of the unit interval $[0,1]$.  As a result, an elusive point can be given an address in terms of $L$'s and $R$'s.  Geometrically, this is motivated by the fact that with each iteration of $\omegati{n}$, one adds to each scaled copy of $\omegati{0}$ two smaller copies of $\omegati{0}$ of scale $2^{-n-1}$, a left copy and a right copy; see Figure \ref{fig:addressingElusivePoint}.  We want to distinguish between what we call \textit{rational} and \textit{irrational} elusive points.  Let $x\in\mathscr{E}$.  Then $x$ is called a \textit{rational elusive point} if the address for $x$ is a preperiodic sequence of $L$'s and $R$'s.  Otherwise, $x$ is called an \textit{irrational elusive point}; see Figure \ref{fig:addressingElusivePoint}.

\begin{figure}
\begin{center}
\includegraphics[scale=.7]{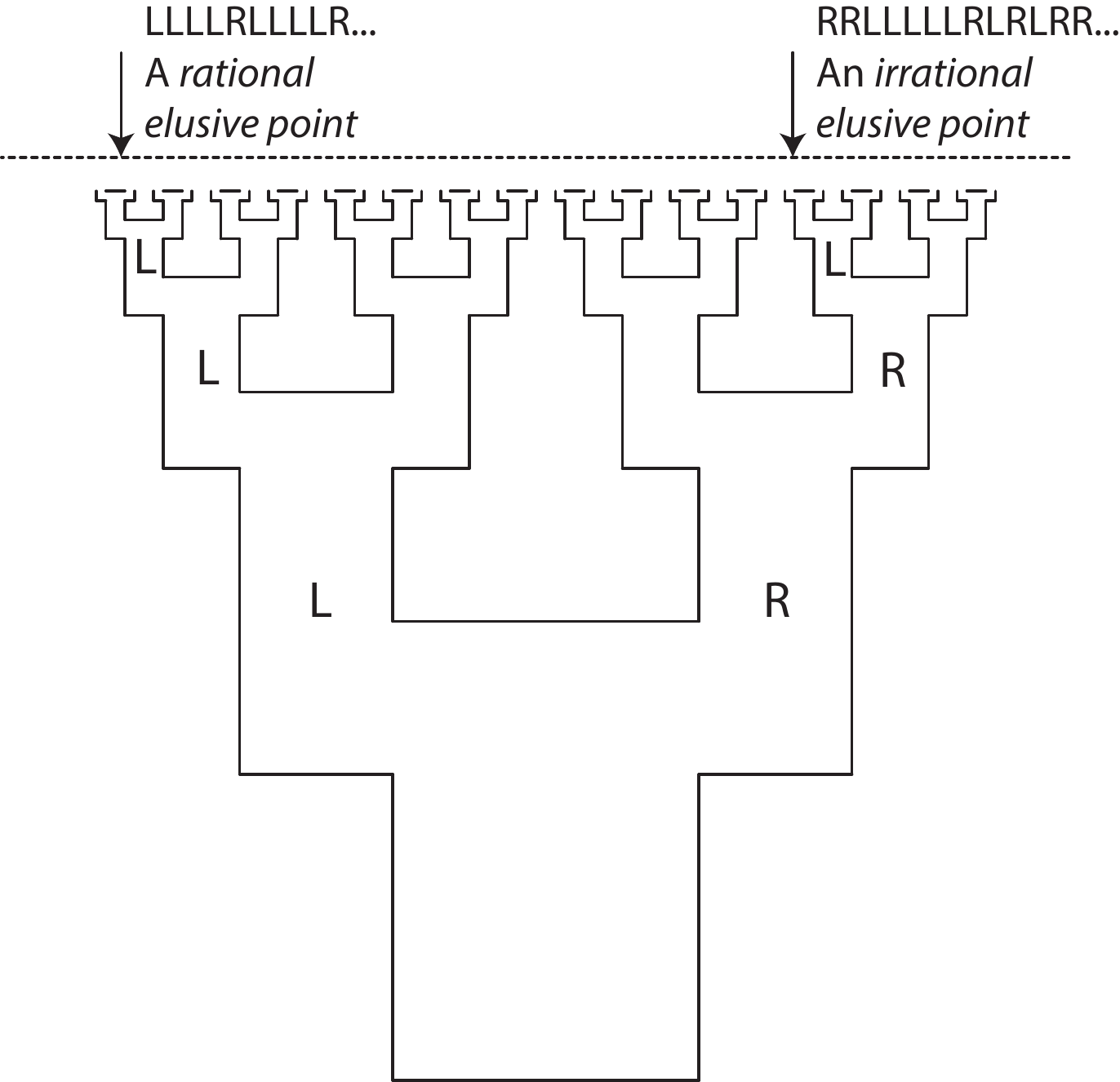}
\end{center}

\caption{An example of how to address a rational elusive point and an irrational elusive point of $\omegat$.  The preperiodic address of the rational elusive point (left) is $\overline{LLLLR}$, whereas the address for the irrational elusive point (right) will not follow any pattern at all.}
\label{fig:addressingElusivePoint}
\end{figure}
\begin{definition}[Compatible initial conditions]
\label{def:compatibleInitialConditions}
Without loss of generality, suppose that $n$ and $m$ are nonnegative integers such that $n > m$. Let $(\xio{n},\theta_n^0)\in (\tfraci{n}\times S^1)/\!\sim$ and $(\xio{m},\theta_m^0)\in (\tfraci{m}\times S^1)/\!\sim$ be two initial conditions of the orbits $\orbiti{n}$ and $\orbiti{m}$, respectively, where we are assuming that $\theta_n^0$ and $\theta_m^0$ are both inward pointing.  If $\theta_n^0 = \theta_m^0$ and if $\xio{n}$ and $\xio{m}$ lie on a segment determined from $\theta_n^0$ (or $\theta_m^0$) which intersects $\omegati{n}$ only at $\xio{n}$, then we say that $(\xio{n},\theta_n^0)$ and $(\xio{m},\theta_m^0)$ are \textit{compatible initial conditions}; see Figures \ref{fig:compInitCondTFractal} and \ref{fig:altCompSeqInitCond}.
\end{definition}

\begin{remark}
When two initial conditions $(\xio{n},\theta_n^0)$ and $(\xio{m},\theta_m^0)$ are compatible, then we simply write them as $(\xio{n},\theta^0)$ and $(\xio{m},\theta^0)$, respectively. If two orbits $\orbiti{m}$ and $\orbiti{n}$ have compatible initial conditions, then we say that such orbits are \textit{compatible}.  Consequently, two compatible orbits $\orbiti{m}$ and $\orbiti{n}$ will now often be written as $\orbitiang{m}{\theta^0}$ and $\orbitiang{n}{\theta^0}$, respectively.  
\end{remark}


\begin{definition}[Sequence of compatible initial conditions]
\label{def:sequenceOfCompatibleInitialConditions}
Let $\{(\xio{n},\theta_n^0)\}_{n=i}^\infty$ be a sequence of initial conditions, for some nonnegative integer $i$.  We say that this sequence is a \textit{sequence of compatible initial conditions} if for every $m\geq i$ and for every $n> m$, we have that $(\xio{n},\theta_n^0)$ and $(\xio{m},\theta_m^0)$ are compatible initial conditions.  In such a case, we then write the sequence as $\{(\xio{n},\theta^0)\}_{n=i}^\infty$.
\end{definition}

In this article, $\xio{n}$ will never be on a segment of $\tfraci{n}$ to be removed in the construction of $\omegati{n+1}$.  Hence, there exists a nonnegative integer $i$ such that $\xio{n} = \xio{i}$ for all $n\geq i$ as a point in the plane.\footnote{We note that for $m\neq n$,  $\xio{n}$ and $\xio{m}$ really lie in two different spaces.  When we say that they are equal, we are actually implying that some embeddings of $\xio{n}$ and $\xio{m}$ into the plane are equal. This is a technical detail that will not cause any problems, but is worth mentioning.} See Figure \ref{fig:T-FractalSequenceOfCompatibleOrbits} for an example of a sequence of compatible orbits.

As was alluded to at the beginning of this section, Definitions \ref{def:compatibleInitialConditions} and \ref{def:sequenceOfCompatibleInitialConditions} were stated in full generality so that the reader may understand the results in the aforementioned articles, should they refer back to them for  examples of periodic orbits and nontrivial paths of other fractal billiard tables (namely, the Koch snowflake fractal billiard and a self-similar Sierpinski carpet fractal billiard table studied, respectively, in [LapNie1--5], \cite{CheNie1}, \cite{LapNie4}).  As we see in Figure \ref{fig:compInitCondTFractal}, it may be that two initial conditions are compatible and $x^0_0 = x^0_n$ for all $n\geq 0$.  Also, it is possible, as shown in Figure \ref{fig:altCompSeqInitCond}, that $\xio{n}\neq \xio{n+1}$ for, in this case, $n=1$, thereby necessitating a more general definition.\footnote{As we mentioned, however, we will always assume that $\xio{i}$ is not on a segment to be removed in the construction of $\omegati{i+1}$ from $\omegati{i}$.}  In the case of the Koch snowflake, we give an example in Figure \ref{fig:kochSnowflakeExample} that demonstrates the necessity for the full generality of Definitions \ref{def:compatibleInitialConditions} and \ref{def:sequenceOfCompatibleInitialConditions}.

\begin{definition}[Sequence of compatible orbits]
\label{def:SequenceOfCompatibleOrbits}
Consider a sequence of compatible initial conditions $\{(\xio{n},\theta^0)\}_{n=i}^\infty$.  Then the corresponding sequence of orbits $\seqi{i}$ is called \textit{a sequence of compatible orbits}.
\end{definition}
\begin{figure}
\begin{center}
\includegraphics[scale=.30]{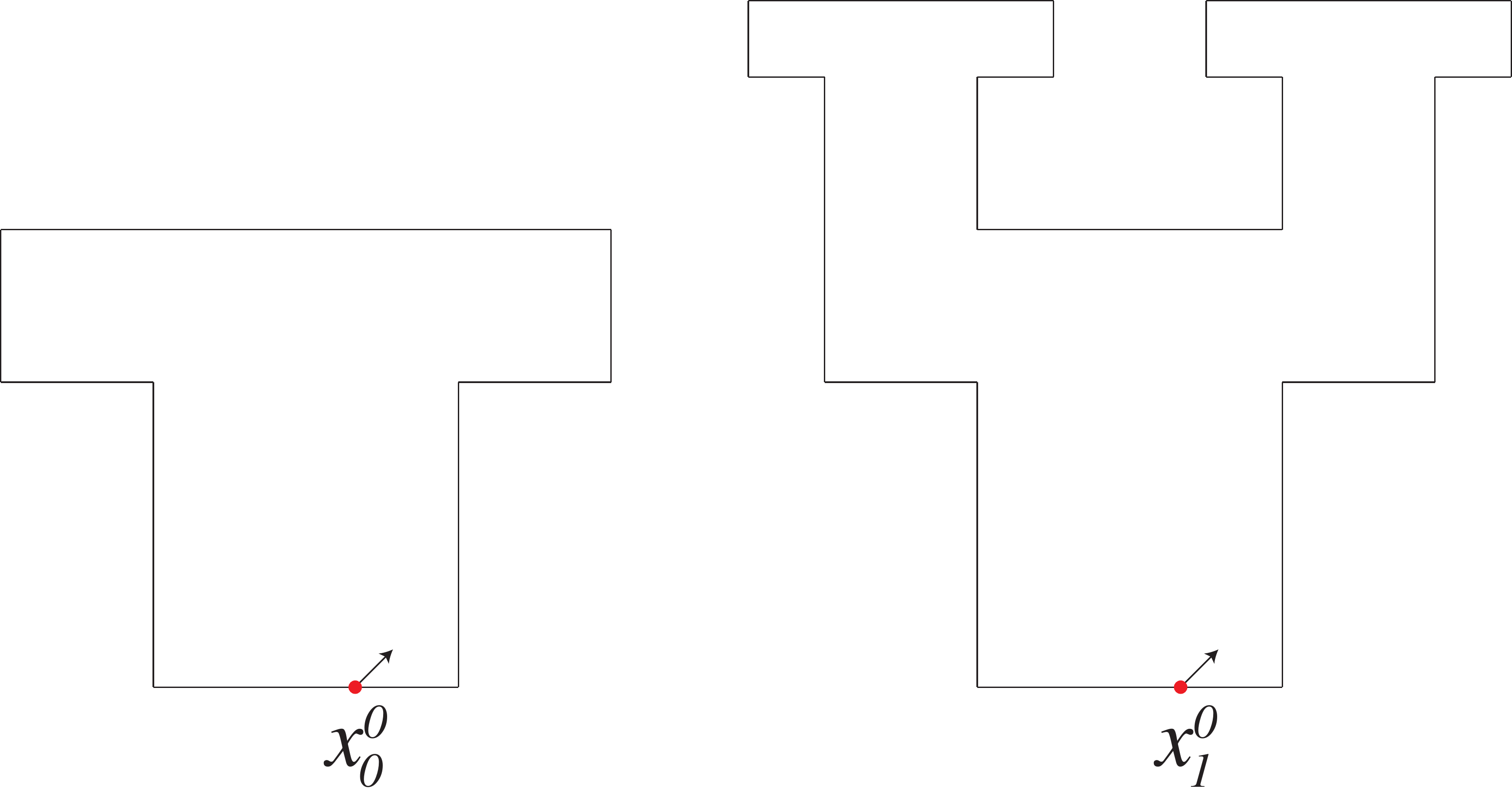}
\end{center}
\caption{Two initial conditions that are compatible initial conditions.  In this case, the initial basepoint of each initial condition is the same point in the plane.}
\label{fig:compInitCondTFractal}
\end{figure}
\begin{figure}
	\begin{center}
		\includegraphics[scale=.30]{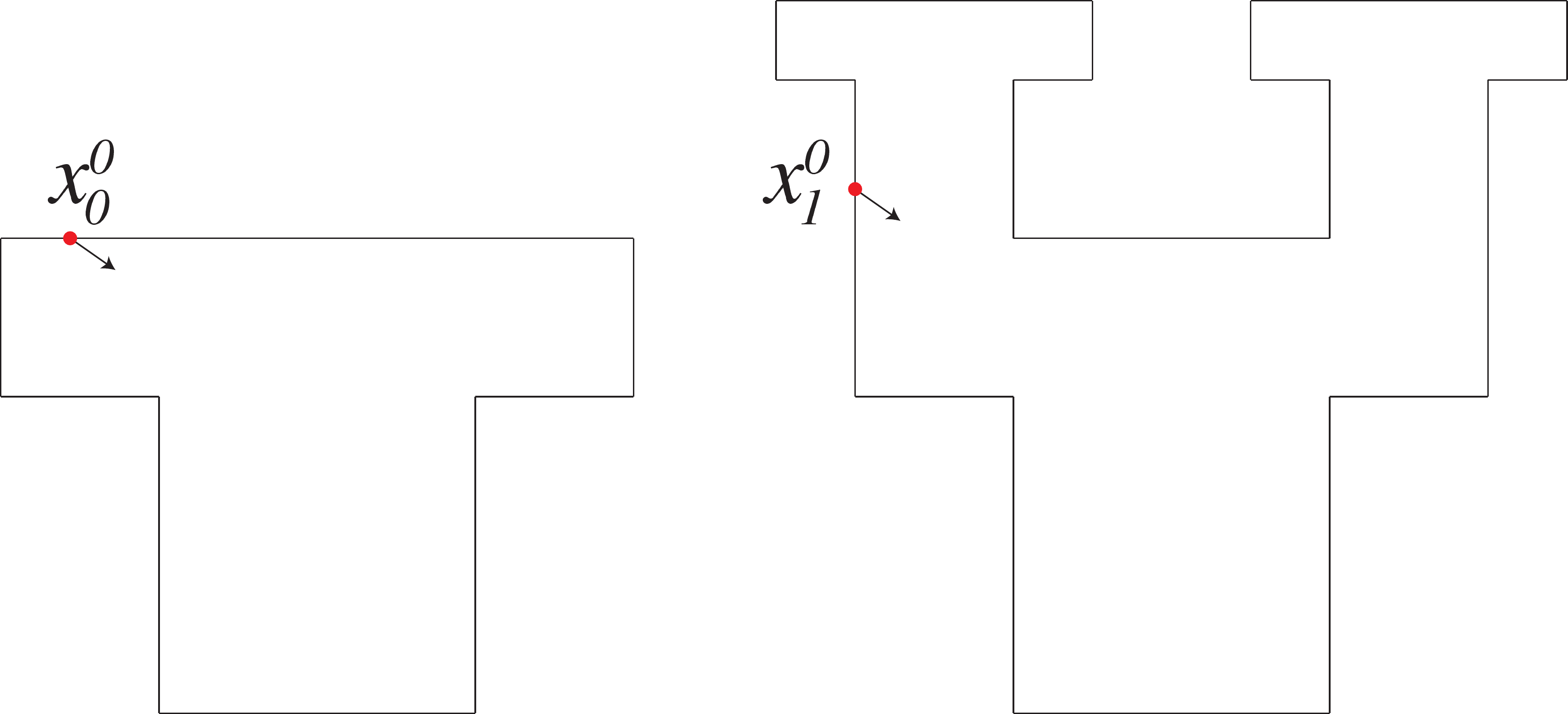}
	\end{center}
	\caption{Two compatible initial conditions for which $\xoo \neq \xio{1}$, but $\xio{n}=\xio{1}$ for all $n\geq 1$.  This case, and those similar to it, will not occur in this article.}
	\label{fig:altCompSeqInitCond}
\end{figure}
\begin{figure}
\begin{center}
	\includegraphics[scale=.20]{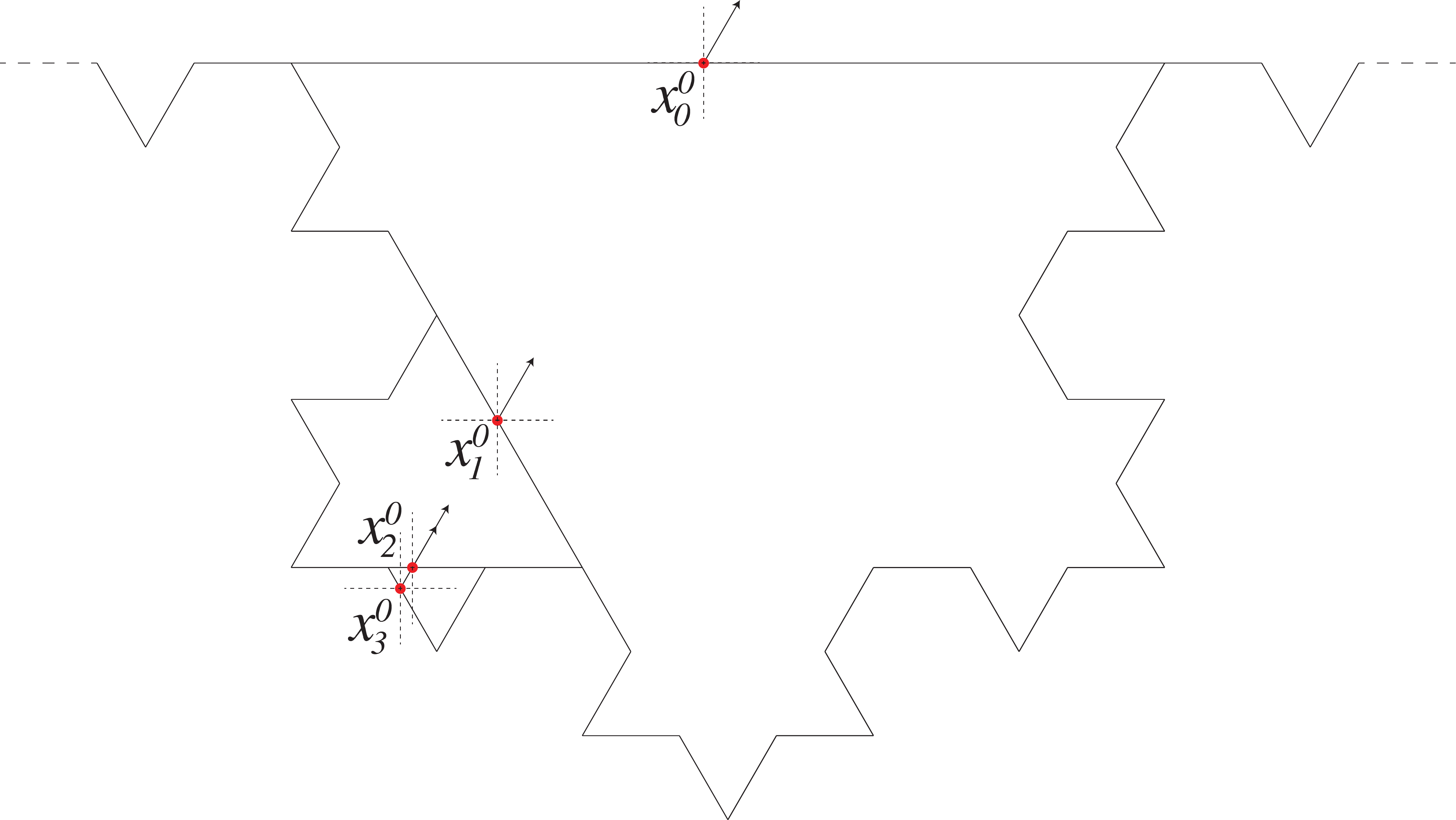}
\end{center}	
\caption{We show here part of an example of a sequence of compatible initial conditions of orbits of prefractal Koch snowflake billiard tables that demonstrates the necessity for the full generality exhibited in Definitions \ref{def:compatibleInitialConditions} and \ref{def:sequenceOfCompatibleInitialConditions}.}
\label{fig:kochSnowflakeExample}
\end{figure}
\begin{figure}
\begin{center}
\includegraphics[scale=.7]{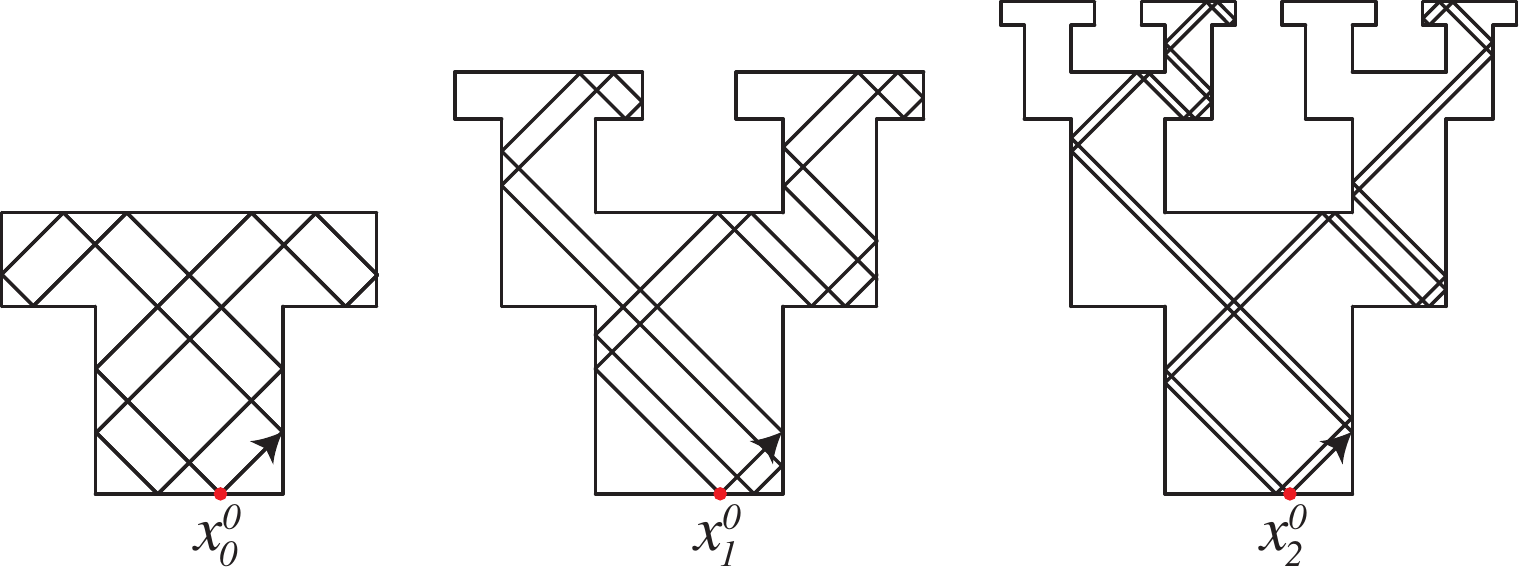}
\end{center}
\caption{A sequence of compatible periodic orbits of $\omegati{0}$, $\omegati{1}$ and $\omegati{2}$, respectively.  Each orbit has the same initial condition $(\frac{2}{3},\theta^0)$, where $\theta^0 = \pi/4$.}
\label{fig:T-FractalSequenceOfCompatibleOrbits}
\end{figure}
If $\orbitixang{m}{\xio{m}}{\theta_m^0}$ is an orbit of $\omegati{m}$, then $\orbitixang{m}{\xio{m}}{\theta_m^0}$ is a member of a sequence of compatible orbits $\seqiang{i}{\theta^0}$, for some $i\leq m$. It is clear from the definition of a sequence of compatible orbits that such a sequence is uniquely determined by the first orbit $\orbitiang{i}{\theta^0}$.  Since the initial condition of an orbit determines the orbit, we can say without any ambiguity that a sequence of compatible orbits is determined by an initial condition $(\xio{i},\theta^0)$.  Consequently, one is then in a position to discuss particular types of sequences of compatible orbits, namely, sequences where, for a given sequence, every orbit in that sequence has the same property.  More precisely, let $\mathcal{P}$ be a property (resp., $\mathcal{P}_1,...,\mathcal{P}_j$ a list of properties).  If every orbit in a sequence of compatible orbits has the property $\mathcal{P}$ (resp., a list of properties $\mathcal{P}_1,...,\mathcal{P}_j$), then we call such a sequence \textit{a sequence of compatible $\mathcal{P}$ \emph{(}resp., $\mathcal{P}_1,...,\mathcal{P}_j$\emph{)}  orbits}.

\begin{definition}[First return time $\upsilon_n$ of an orbit $\orbiti{n}$]
\label{def:firstReturnTime}
Let $\orbiti{n}$ be an orbit of $\omegati{n}$ with $\xio{n}$ on some side $\sigma$ of $\mathscr{T}_n$.  Then, the least positive integer $k$ such that $\xii{n}{k}$ lies on $\sigma$ is called the \textit{first return time} of the orbit and is denoted by $\upsilon_n$.
\end{definition}

In general, for any $j\geq 1$, we let $\upsilon_{n}^j$ be the $j$th return time of  the orbit $\orbiti{n}$.  As previously mentioned, given a sequence of compatible orbits $\seqi{i}$, the geometry of the $T$-fractal dictates that there exists $m\geq i$ such that $\xio{n} = \xio{m}$ for every $m\geq n$.  This implies that there exists $m\geq i$ such that for every $j\geq 1$, $\xii{n}{\upsilon^j_n}$ is on the same segment as $\xio{m}$ for all $n\geq m$.  

\begin{notation}
We denote by $\orbiti{n}_{\upsilon_n}$ the portion of the orbit given by $\{(\xii{n}{k},\theta_n^{k})\}_{k = 0}^{\upsilon_n}$.  At times, $\orbiti{n}_{\upsilon_n}$ may also denote the path connecting the points $\{\xii{n}{k}\}_{k = 0}^{\upsilon_n}$.  It will be clear from the context whether $\orbiti{n}_{\upsilon_n}$ is either viewed as a path or as a collection of elements in the phase space.

\label{not:nontrivialPathInOrbitNotation}
\end{notation}

\begin{definition}[First escape time $\tau_n$ of an orbit $\orbiti{n}$]
\label{def:firstEscapeTime}
Let $\orbiti{n}$ be an orbit of $\omegati{n}$ with $\xio{n}$ on some side $\sigma$ of $\mathscr{T}_n$, where $\xio{n}$ is not lying on a segment of $\sigma$ to be removed in the construction of $\tfraci{n+1}$ from $\tfraci{n}$.  Then, the least positive integer $k$ such that $\xii{n}{k}$ lies on a segment $\sigma'$ to be removed in the construction of $\omegati{n+1}$ from $\omegati{n}$ is called the \textit{first escape time} and is denoted by $\tau_n$.  In the event an orbit does not intersect with a segment to be removed in the construction of $\omegati{n+1}$ from $\omegati{n}$, the first escape time will be defined to be infinity.  
\end{definition}

\begin{notation}
We denote by $\orbiti{n}_{\tau_n}$ the portion of the orbit given by $\{(\xii{n}{k},\theta_n^{k})\}_{k = 0}^{\tau_n}$.  At times, $\orbiti{n}_{\tau_n}$ may also denote the path connecting the points $\{\xii{n}{k}\}_{k = 0}^{\tau_n}$.  It will be clear from the context whether $\orbiti{n}_{\tau_n}$ is viewed as a path or as a collection of elements in the phase space.

\end{notation}

\begin{figure}
\begin{center}
\includegraphics[scale = .80]{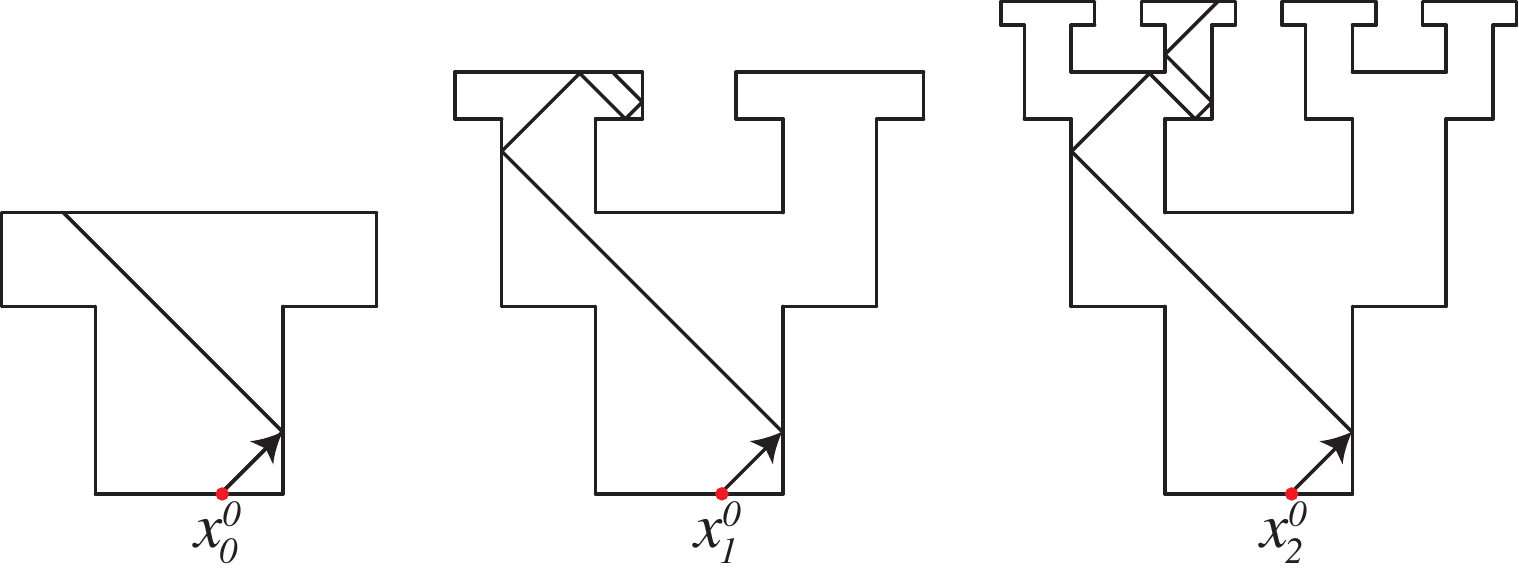}
\end{center}
\caption{An example of three paths from a sequence of paths reaching an elusive point of $\omegat$.  From left to right, the paths are determined from  $\orbitiang{0}{\theta^0}_{\tau_0},\orbitiang{1}{\theta^0}_{\tau_1},\orbitiang{2}{\theta^0}_{\tau_2}$.  These paths are also denoted by $\ntpi{0},\ntpi{1},\ntpi{2}$, respectively; see Definition \ref{def:ANontrivialPath}.}
\label{fig:portionOfOrbitHittingSegmentRemoved}
\end{figure}


Let $\seqi{i}$ be a sequence of compatible orbits such that  $\tau_n <\infty$ for every $n\geq i$.  From now on, we denote by $\ntpi{n}$ the path given by $\orbitiang{n}{\theta^0}_{\tau_n}$.  We next define what it means for a path in $\omegat$ to be a \textit{nontrivial path}. (Recall that, intuitively, a nontrivial path in $\omegat$ is one that manages to never remain confined in any finite approximation $\omegati{n}$ of the billiard table $\omegat$.)

\begin{definition}[Nontrivial path]   
\label{def:ANontrivialPath}
Let $\seqi{i}$ be a sequence of compatible orbits and $\ntpi{n}$  as defined in the preceding paragraph.  Then, $\overline{\bigcup_{n=i}^\infty \ntpi{n}}$ is called a \textit{nontrivial path of $\omegat$} and is  denoted by $\ntp$, where $x^0 := \xio{i}$.
\end{definition}


 In Figure \ref{fig:portionOfOrbitHittingSegmentRemoved}, the first three paths $\ntpi{i}$, $i=0,1,2$, of a sequence $\{\ntpi{n}\}_{n=0}^\infty$ derived from a particular sequence of compatible periodic orbits, are given. 

\begin{remark}
Definition \ref{def:ANontrivialPath} is more precise than Definition 5.5 given in the article  \cite{LapNie4}.  The notation of Definition \ref{def:ANontrivialPath} reflects the fact that the T-fractal billiard is the subject of the present paper, but is still general enough to be applicable to the case of the Koch snowflake fractal billiard.
\end{remark}

As was mentioned just after Definition \ref{def:firstReturnTime}, we denote by $\upsilon_n^j$ the $j$th return time of the orbit $\orbiti{n}$.  We now define what it means for an orbit of $\omegat$ to be a \textit{recurrent orbit} of the $T$-fractal billiard table.  We also note that Definitions \ref{def:recurrentOrbit}--\ref{def:ASingularOrbit} below are specific to the $T$-fractal billiard.  However, it is plausible that one may be able to appropriately generalize them so as to account for what have been called \textit{periodic orbits} of the given fractal billiard table(s) in each of \cite{CheNie1,LapNie1,LapNie2,LapNie3,LapNie4}.  Moreover, in some of the aforementioned works, we have indicated that a periodic orbit of a fractal billiard table may be the suitable limit of a sequence of compatible periodic orbits, namely the Hausdorff--Gromov limit.\footnote{\label{ftnote:HausdorffGromovLimit}We have discussed the use of the Hausdorff--Gromov limit since, technically speaking, each orbit is in an entirely different space, though each may project down onto the plane and into a billiard table that contains the previous prefractal billiard table approximation.}  In the results presented in the subsequent sections, we will see how the Hausdorff--Gromov limits of particular sequences of compatible orbits do in fact produce orbits of the $T$-fractal billiard satisfying Definition \ref{def:periodicOrbit}.  Definitions \ref{def:recurrentOrbit}--\ref{def:ASingularOrbit} remain more general in the event that the Hausdorff--Gromov limit of a sequence of compatible periodic orbits cannot be calculated or does not exist.  In the latter case, we may still see that such (perhaps, pathological) examples satisfy Definition \ref{def:periodicOrbit}, but exhibit some oddity that prevents them from being the Hausdorff--Gromov limit of a sequence of compatible orbits.

\begin{definition}[Recurrent orbit of $\omegat$]
Let $\seqi{i}$ be a sequence of compatible orbits such that $\xio{m} = \xio{i}$ for all $m\geq i$.\footnote{Recall that the geometry of the $T$-fractal billiard table makes this possible, as discussed in the text just after Definition \ref{def:sequenceOfCompatibleInitialConditions} and the text prior to Notation \ref{not:nontrivialPathInOrbitNotation}.}  If for every $j\geq 1$, we have that $\{\xii{n}{\upsilon^j_n}\}_{n=i}^\infty$ converges to some $x$ in the segment containing $\xio{i}$, then we denote $x$ by $x^{\upsilon^j}$ and call $\upsilon^j$ the $j$th return time of the orbit $\billiardOrbit$ and $\billiardOrbit$ a \textit{recurrent orbit} of $\omegat$.	
\label{def:recurrentOrbit}
\end{definition}

Figure \ref{fig:T-FractalSequenceOfCompatibleOrbits} gives an illustration of a sequence of points $\{\xii{n}{\upsilon_n}\}_{n=0}^\infty$ converging to some $x$ in the unit-interval base of $\omegat$.  Moreover, we consider $x$ to be the point of first return for the orbit of the fractal billiard table.  In general, it does not have to be the case that $\xii{n}{\upsilon_n^j} = x^{\upsilon^j}$ for sufficiently large $n$ and all $j\geq 1$.  In fact, in many of the examples and results that we will present, it is the case that $\xii{n}{\upsilon_n} \neq \xii{m}{\upsilon_m}$, whenever $n\neq m$; see, for example, Figure \ref{fig:T-FractalSequenceOfCompatibleOrbits}.  We now formally define what it means for an orbit to be a \textit{periodic orbit} of $\omegat$.

\begin{definition}[Periodic orbit of $\omegat$]
\label{def:periodicOrbit}
Let $\billiardOrbit$ be a recurrent orbit.  If $\{x^{\upsilon^j}\}_{j=1}^\infty$ is as defined in  Definition \ref{def:recurrentOrbit} and is a periodic sequence with finite period, then we say that $\billiardOrbit$ is a \textit{periodic orbit} of $\omegat$.	
\end{definition}

The three orbits shown in Figure \ref{fig:T-FractalSequenceOfCompatibleOrbits} are part of a sequence of compatible orbits satisfying Definitions \ref{def:recurrentOrbit} and  \ref{def:periodicOrbit}.  As one may be able to deduce from Figure \ref{fig:T-FractalSequenceOfCompatibleOrbits}, $\xii{n}{\upsilon_n} \neq \xii{m}{\upsilon_m}$ for all $n\neq m$, but $\xii{n}{\upsilon_n} \to x^{\upsilon}$ as $n\to \infty$ where $x^{\upsilon}=x^0$ and $x^0 := \xoo$. The sequence of compatible orbits illustrated in Figure \ref{fig:T-FractalSequenceOfCompatibleOrbits} is a sequence of compatible orbits where $\orbitixang{n}{\xio{n}}{\theta^0} \neq \orbitixang{m}{\xio{m}}{\theta^0}$ for all $n\neq m$.  On the other hand, Definitions \ref{def:recurrentOrbit} and \ref{def:periodicOrbit} are phrased in such a way as to also account for what we will see is the trivial limit of an eventually constant sequence of compatible periodic orbits; see Definition \ref{def:EventuallyConstanceSequenceOfCompatibleOrbits} and Theorem \ref{thm:eventuallyConstantSequenceOfCompatOrbits}.

When we speak of a singular orbit of $\omegati{n}$, we mean that the forward orbit is a singular orbit.  This is to differentiate from what are called saddle connections.\footnote{A saddle connection of a polygonal billiard table is an orbit that connects two singularities, or corners, of the billiard table.}  The reason for this distinction is that an orbit with $\theta^0$ such that $\tan\theta^0$ is irrational may yield a singular forward orbit, yet starting from the same point $\xio{n}$, but with a new direction of $\pi - \theta^0$, will yield a dense orbit, or vice-versa.

\begin{definition}[Singular orbit of $\omegat$]

Consider a sequence of compatible orbits $\seqi{i}$.  Then, we say that $\billiardOrbit$, where $x^0 = \xio{i}$, is a \textit{singular orbit} of $\omegat$ in either of the following two cases:

\begin{enumerate}
	\item For some $j\geq 1$, the sequence $\{x_n^{\upsilon_n^j}\}_{n=i}^\infty$ has more than one accumulation point.
\vspace{2 mm}
	\item  $\seqi{i}$ is a sequence of compatible singular orbits (i.e, where each orbit is a singular orbit of its respective prefractal approximation).
\end{enumerate}
\label{def:ASingularOrbit}
\end{definition}

In the classical case of a rational billiard table, when a pointmass intersects with a corner where the orbit cannot be continued in a well-defined manner, we terminate the trajectory.  It is very ambiguous how a pointmass would continue in such a situation.  In Part (1) of Definition \ref{def:ASingularOrbit}, the fact that, for some $j\geq 1$, the sequence $\{x_n^{\upsilon^j_n}\}_{n=i}^\infty$ does not converge to a single limit point highlights an analogous ambiguity for the pointmass in the fractal billiard table.  In \S\ref{sec:aNontrivialPathInAnIrrationalDirection}, we give an example of a sequence of compatible orbits satisfying Part (2) of Definition \ref{def:ASingularOrbit}.  While no example of a sequence of compatible orbits satisfying Part (1) has yet been found, this does not mean that such a definition is unnecessary.  Indeed, this definition captures the essence of a singular orbit, namely, that such an orbit cannot be continued in a well-defined manner.  For each $j\geq 1$, the set $\{\xii{n}{\upsilon_n^j}\}_{n=i}^\infty$ having two or more accumulation points is analogous to an orbit having two or more ways of continuing past a nonremovable singularity (i.e., one for which reflection cannot be well defined).  We believe that such an example can be found, this being the focus of \cite{JohNie2}.



\section{Sequences of compatible periodic orbits}
\label{sec:SequencesOfCompatiblePeriodicOrbits}
In this section, we give sufficient conditions for when a sequence of compatible orbits is a sequence of compatible \textit{periodic} orbits.  Our goal here is to lay the foundation for the study of the limiting behavior of particular sequences of compatible periodic orbits considered in \S\ref{sec:NontrivialPathsInTheTFractalBilliardTable}.  In \S\ref{sec:aNontrivialPathInAnIrrationalDirection},  we will give an example of a sequence of compatible singular orbits that yields a nontrivial path of the $T$-fractal billiard table.  Such a nontrivial path will constitute a singular orbit of $\omegat$.

We remind the reader that the geometry of $\omegati{0}$ is given in Figure \ref{fig:T0}.  Determining which intercepts and slopes yield line segments in the plane that avoid lattice points of the form $(\frac{a}{2^c},\frac{b}{2^d})$, with $c$ and $d$ nonnegative integers, is equivalent to specifying an initial condition of an orbit of a square billiard table that avoids corners of the billiard table.  Extending this reasoning to $\omegat$, we can determine various sufficient conditions for the existence of a sequence of compatible \textit{periodic} orbits of the prefractal billiard tables $\omegati{n}$, for $n\geq 0$.

Specifically, using the fact that an appropriately scaled square billiard table tiles $\omegati{n}$, we can reflect-unfold such an orbit in $\omegati{n}$ in order to determine an orbit of $\omegati{n}$.  

\begin{proposition}
\label{thm:WhatTheSlopeCannotBeInTFrac}
Let $\xio{0} = \frac{t}{h^k}$, with  $k,t$ being positive integers, $t$ and $h$ relatively prime, $h$ a positive odd integer and $0<t<h^k$.  Furthermore, let $m\in\R$.  If for every $p,q,r,s\in \Z$, $r,s\geq 0$, we have that
\begin{align}
m &\neq \frac{q2^{r-s}h^k}{ph^k-t2^r},
\label{eqn:WhatmCannotBe}
\end{align}
\noindent then the line $y=m(x-\xio{0})$ does not contain any point of the form $(\frac{a}{2^c},\frac{b}{2^d})$, $a,b,c,d\in \Z$, with $c,d\geq 0$.
\end{proposition}

Note that the condition (\ref{eqn:WhatmCannotBe}) above is automatically satisfied if the slope $m$ is irrational.

\begin{proof}
Suppose there exist $a,b,c,d\in \Z$, $c,d\geq 0$, such that
\begin{align}
\notag\frac{b}{2^d} = m\left(\frac{a}{2^c}-\xio{0}\right).
\end{align}
Then, after a few algebraic manipulations, we obtain that
\begin{align}
\notag 2^c\frac{b}{2^d} &= m(a-2^c\xio{0})
\end{align}
\noindent or, equivalently, 
\begin{align}
\notag\frac{h^k2^{c-d}b}{ah^k-t2^c} &=m,
\end{align}
\noindent which clearly contradicts hypothesis (\ref{eqn:WhatmCannotBe}).
\end{proof}

\begin{proposition}
\label{thm:WhatTheSlopeCanBeInTFrac}
Let $\xio{0} = \frac{t}{h^k}$, with $k,t$ being positive integers, $t$ and $h$ relatively prime, $h$ a positive odd integer and $0<t<h^k$.  If
\begin{align}
\notag m&=\frac{2^\gamma}{(2\alpha +1)^\beta},
\end{align}
\noindent with $\alpha,\beta,\gamma$ being nonnegative integers, then, for every $p,q,r,s\in \Z$ with $r,s\geq 0$, the point $(\frac{p}{2^r},\frac{q}{2^s})$ does not lie on the line $y=m(x-\xio{0})$.
\end{proposition}

\begin{proof}
Suppose there exist $p,q,r,s\in \Z$ with $r,s\geq 0$ such that
\begin{align}
\notag\frac{q}{2^s} &=m\left(\frac{p}{2^r} - \xio{0}\right).
\end{align}
\noindent Then, after various algebraic manipulations, we arrive at
\begin{align}  
h^k(2^{\gamma+s-r}p-q(2\alpha+1)^\beta) &= t2^{\gamma+s}.
\label{eqn:aCalcShowingOddIsNotEven}
\end{align}
\noindent Since $k,t>0$ and $(t,h)=1$ (i.e., $t$ and $h$ are relatively prime), we see that the left-hand side of Equation (\ref{eqn:aCalcShowingOddIsNotEven}) contains a factor of $h$ and the right-hand side of Equation (\ref{eqn:aCalcShowingOddIsNotEven}) does not.  This is a contradiction.  Hence, the point $(\frac{p}{2^r},\frac{q}{2^s})$ does not lie on the line $y=m(x-\xio{0})$.
\end{proof}

Finally, Propositions \ref{thm:WhatTheSlopeCannotBeInTFrac} and \ref{thm:WhatTheSlopeCanBeInTFrac} (combined with the fact that an initial condition of an orbit of $\omegati{i}$, $i\geq 0$, determines a sequence of compatible orbits $\seqi{i}$), allow us to determine a countably infinite family of sequences of compatible periodic orbits.  We state this as a theorem.

\begin{theorem}
\label{thm:sequenceOfCompatibleNonsingularOrbitsPeriodicIfRational}
Let $\xoo$ be a point in the unit interval and $\theta^0$ be such that $m=\tan\theta^0$.  If $\xoo$ and $m$ satisfy Propositions \ref{thm:WhatTheSlopeCannotBeInTFrac} or \ref{thm:WhatTheSlopeCanBeInTFrac}, then $\seqi{0}$ is a sequence of compatible orbits for which each orbit is nonsingular in its respective billiard table.  Moreover, if $m$ is rational, then $\seqi{0}$ is a sequence of compatible \textit{periodic} orbits.
\end{theorem}

\begin{proof}[Sketch of the proof]
	In both Propositions \ref{thm:WhatTheSlopeCannotBeInTFrac} and  \ref{thm:WhatTheSlopeCanBeInTFrac}, the idea is to construct a family of point-slope pairs such that any line with such a slope and passing through the given point necessarily avoids all points of that plane that would have corresponded to any dyadic point.  If one tiles the plane by squares with side-length $\frac{1}{2^k}$, $k\geq 0$, then any line that avoids dyadic points of the plane necessarily avoids corners of any square with side-length $\frac{1}{2^k}$ tiling the plane.  Since for all $k\geq 0$, $\omegati{k}$ is tiled by squares of side-length $\frac{1}{2^{k+1}}$, we have that any orbit $\orbitixang{k}{\xio{k}}{\theta^0}$ with $\xio{k}$ and $m = \tan \theta^0$  satisfying the hypotheses of Propositions \ref{thm:WhatTheSlopeCannotBeInTFrac} and \ref{thm:WhatTheSlopeCanBeInTFrac} will avoid corners of $\omegati{k}$.  Thus, $\seqixang{0}{\xio{n}}{\theta^0}$ will be a sequence of compatible orbits.  Moreover, if $m=\tan \theta^0$ is rational, each orbit $\orbitixang{n}{\xio{n}}{\theta^0}$ will be periodic in $\omegati{n}$.  Therefore, when $m$ is rational, $\seqixang{0}{\xio{n}}{\theta^0}$ will be a sequence of compatible \textit{periodic} orbits.
\end{proof}

\begin{example}
\label{exa:T-fractalSequenceOfCompatibleOrbits}
Let $\xio{0} = \frac{2}{3}$ and $\theta^0 = \frac{\pi}{4}$.  Then $\seqang{\frac{\pi}{4}}$ is a sequence of compatible periodic orbits; see Figure \ref{fig:T-FractalSequenceOfCompatibleOrbits}.
\end{example}

We now introduce a specific family of sequences of compatible periodic orbits.  A sequence of compatible periodic orbits in such a family will clearly have a trivial limit, this point being discussed at the end of \S\ref{sec:NontrivialPathsInTheTFractalBilliardTable}.

\begin{definition}[Eventually constant sequence of compatible orbits]
\label{def:EventuallyConstanceSequenceOfCompatibleOrbits}
Consider the sequence of compatible orbits $\seqi{i}$.  If there exists a nonnegative integer $N$ such that for every $n\geq N$, the points for which $\orbitiang{n}{\theta^0}$ intersects $\omegati{n}$  are the same points for which $\orbitiang{N}{\theta^0}$ intersects $\omegati{N}$, then we say that $\seqi{i}$ is an \textit{eventually constant sequence of compatible orbits}.  Furthermore, a sequence of compatible orbits $\seqi{i}$ is an eventually constant sequence of compatible orbits if and only if $\seqiang{i}{\pi - \theta^0}$ is an eventually constant sequence of compatible orbits.\footnote{Recall that we are requiring that $\xio{n} = \xio{i}$ for all $n\geq i$.}
\end{definition}

It should be made clear that a sequence of compatible dense orbits given by $\seqi{i}$ will never be an eventually constant sequence of compatible orbits, because we are concerned with both the forward orbit and the backwards orbit simultaneously, though the notation may belie this fact; this is another subtle difference between an orbit and a nontrivial path.  Suppose  $\seqi{i}$ yields a singular orbit of $\omegat$, in the sense that each (forward) orbit $\orbitiang{n}{\theta^0}$ is a singular orbit.  Instead, now consider the sequence of compatible orbits $\seqiang{i}{\pi-\theta^0}$.  Such a sequence of compatible orbits is never eventually constant, since each orbit must be uniformly distributed in its respective billiard table.  We will see an example of this phenomenon in \S\ref{sec:aNontrivialPathInAnIrrationalDirection}.

The following two lemmas are necessary for establishing sufficient conditions for the existence of a particular family of eventually constant sequences of compatible periodic orbits of prefractal billiard tables.

\begin{lemma}
\label{lem:eventuallyConstantStartingFrom1-2}
Suppose an orbit  of $\omegati{0}$ has an initial condition $(\xoo,\theta^0)$ such that $\xoo = \frac{1}{2}$ is an element of the base of $\omegati{0}$ and $\tan\theta^0 = 2^{-n}$ for some integer $n\geq 2$.  Then, the sequence of compatible orbits $\seqii{0}{k}$ is an eventually constant sequence of compatible periodic orbits.
\end{lemma}

\begin{proof}

We proceed by discussing the orbits of a square billiard table with the same initial conditions or related initial conditions.  

Let $\Omega(S)$ be the unit square billiard table and $\Omega(S')$ be the square billiard table with side-length $2^{-1}$.  Unfolding the orbit $\mathscr{O}_S((2^{-1},0),\theta^0)$ in a tiling of the plane by the unit square $S$ results in a straight-line path that intersects $2^n + 1$ squares before reaching the point $(2^n+2^{-1},1)$ in the plane.  During the unfolding process, $2^n$ many unfoldings were made to produce the straight-line path.  Hence, the billiard ball intersects the top of the unit square with a direction that is identical to $\theta^0$.  

Now, consider an orbit $\mathscr{O}_{S'}((0,0),\theta^0)$ of the smaller square billiard table $\Omega(S')$.  Such an orbit intersects the top left singularity of $S'$.  Hence, the reflected-unfolding of $\orbitixang{S'}{(0,0)}{\theta^0}$ embedded in $\Omega(S)$ such that $(0,0)$ in $S'$ corresponds to $(2^{-1},0)$ in $S$ will intersect the midpoint of the top of the billiard $\Omega(S)$.  Since $\omegati{0}$ is tiled by $S'$, we see that the orbit will continue and intersect the midpoint of the top of $\omegati{0}$. By symmetry, the reflected-unfolded orbit does not intersect any segment removed in subsequent approximations.  In addition, the orbit $\orbitixang{0}{(2^{-1},0)}{\theta^0}$ in $\omegati{0}$ remains fixed for every subsequent approximation.
\end{proof}

\begin{example}
In Figure \ref{fig:exampleOfOrbitDescribedInLemma}, we see an example of an orbit of $\omegati{0}$ described in Lemma \ref{lem:eventuallyConstantStartingFrom1-2}.  
\end{example}

\begin{figure}
\begin{center}
\includegraphics[scale=2]{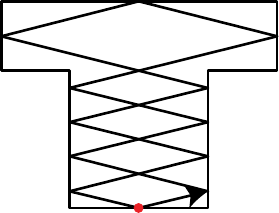}
\end{center}

\caption{An example of an orbit described in Lemma \ref{lem:eventuallyConstantStartingFrom1-2}.  The initial condition of this orbit is $((1/2,0),\theta^0)$, where $\tan\theta^0 = 1/4$.  This orbit is the first orbit in an eventually constant sequence of compatible orbits.}
\label{fig:exampleOfOrbitDescribedInLemma}
\end{figure}

\begin{lemma}
\label{thm:intersectsMidpointOfTop}
Let $k$ be a positive integer.  Consider $\xio{k}$, the midpoint of a lower horizontal segment overhanging the square stump of $\omegati{k}$.  If $\theta^0$ is such that $\tan \theta^0 = 2^{-n}$ for some  integer $n\geq 2$, then $\orbitiang{k}{\theta^0}$ is an orbit which remains in a rectangular region of $\omegati{k}$ and intersects the top of $\omegati{k}$ at the midpoint of a segment removed in the construction of $\omegati{k+1}$ from $\omegati{k}$.
\end{lemma}

\begin{proof}
It suffices to prove the statement for the case $k=1$.  Consider $\xio{1}$, the midpoint of a lower horizontal segment overhanging the square stump of $\omegati{0}$. We know from the proof of the previous lemma that the orbit $\orbitixang{0}{2^{-1}}{\theta^0}$ is an orbit of $\omegati{0}$ that intersects the top of the unit square billiard table at the midpoint by forming a segment with slope $2^{-n}$.  Additionally, $2n$ many reflections are required to reach the point $(2^n+2^{-1},1)$ when unfolding $\orbitixang{0}{2^{-1}}{\theta^0}$ in a tiling of the plane by the unit square.  Since the rectangular region of $\omegati{0}$ is tiled by four squares, each with side-length $\frac{1}{2}$ (in general, side-length $2^{-k}$) and $2^n+2^{-1}\mod{4} = \frac{1}{2}$, because $n\geq 2$, it follows that the reflected-unfolding of $\orbitiang{k}{\theta^0}$ intersects a midpoint of a segment of $\omegati{1}$ removed in the construction $\omegati{2}$.

Since the orbit $\orbitixang{0}{2^{-1}}{\theta^0}$ intersects the base of the square only at the initial basepoint, the reflected-unfolding of such an orbit must do the same in the rectangular region of $\omegati{1}$.

Now, for every $k\geq 1$, let $\xio{k}$ be a midpoint of a segment overhanging a square stump of $\tfractal$ scaled by $2^{-k}$.  An orbit $\orbitiang{k}{\theta^0}$, with the above integer $n$ satisfying $n\geq 2$, will be an orbit that remains in the rectangular region of $\omegati{k}$ containing $\xio{k}$.
\end{proof}

\begin{example}
In Figure \ref{fig:orbitRemainsInRectangle}, we see an example of an orbit of $\omegati{0}$ described in Lemma \ref{thm:intersectsMidpointOfTop}.  
\end{example}

\begin{figure}
\begin{center}
\includegraphics[scale=.4]{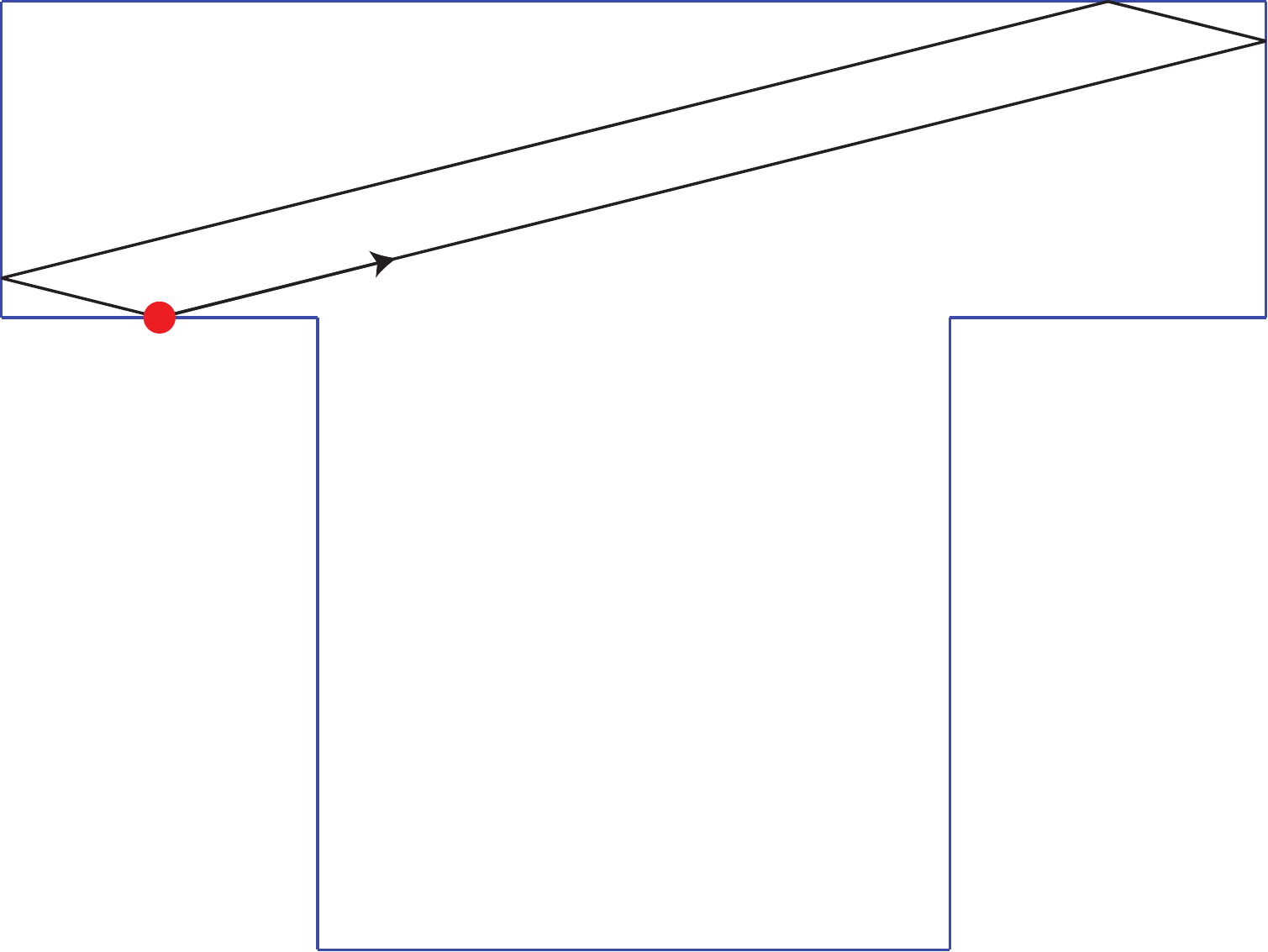}
\end{center}
\caption{This orbit, with the initial condition shown, never exits the rectangular region of $\omegati{0}$.  In other words, it never intersects the base of $\omegati{0}$, but will very clearly escape to $\omegati{1}$. Figure \ref{fig:ConstantOrbitInT1} illustrates how one can construct an eventually constant sequence of compatible orbits.}
\label{fig:orbitRemainsInRectangle}
\end{figure}

\begin{theorem}
\label{thm:eventuallyConstantSequenceOfCompatOrbits}
Let $k$ be a positive integer.  Let $\xio{k}$ be the midpoint of a segment overhanging the square stump of a copy of $\mathscr{T}_0$ of scale $2^{-k}$.  If $\theta^0$ is such that $\tan\theta^0 = 2^{-N}$ for some integer $N\geq 2$, then the sequence of compatible orbits $\seqiang{k}{\theta^0}$ is an eventually constant sequence of compatible periodic orbits.
\end{theorem}

\begin{proof}
It suffices to prove the statement for a particular case.  Let $\xio{0}$ be the midpoint of a segment overhanging the square stump of $\tfraci{0}$.  By Lemma \ref{thm:intersectsMidpointOfTop}, the orbit $\orbitiang{0}{\theta^0}$ intersects the midpoint of a segment removed in the construction of $\omegati{1}$ in such a way that the compatible orbit $\orbitiang{1}{\theta^0}$ enters into the scaled copy of $\mathscr{T}_0$ at such a point and with such a direction that the portion of the orbit contained in the scaled copy of $\mathscr{T}_0$ is, in fact, a scaled copy of the orbit $\orbitixang{0}{2^{-1}}{\theta^0}$ of $\omegati{0}$.  By Lemma \ref{lem:eventuallyConstantStartingFrom1-2}, such an orbit remains fixed in $\omegati{0}$.  Hence, the scaled copy of the orbit $\orbitixang{0}{2^{-1}}{\theta^0}$ of $\omegati{0}$ will coincide with part of the orbit $\orbitiang{1}{\theta^0}$, meaning that $\seqiang{k}{\theta^0}$, with $\theta^0$ such that $\tan\theta^0 = 2^{-N}$ for some $N\geq 2$, is an eventually constant sequence of compatible periodic orbits.
\end{proof}

As discussed in the caption of Figure \ref{fig:orbitRemainsInRectangle}, Figure \ref{fig:ConstantOrbitInT1} illustrates the construction of an eventually constant sequence of compatible periodic orbits.  Such a sequence of compatible periodic orbits is guaranteed to exist by Theorem \ref{thm:eventuallyConstantSequenceOfCompatOrbits}.

\begin{figure}
\begin{center}
\includegraphics[scale = .4]{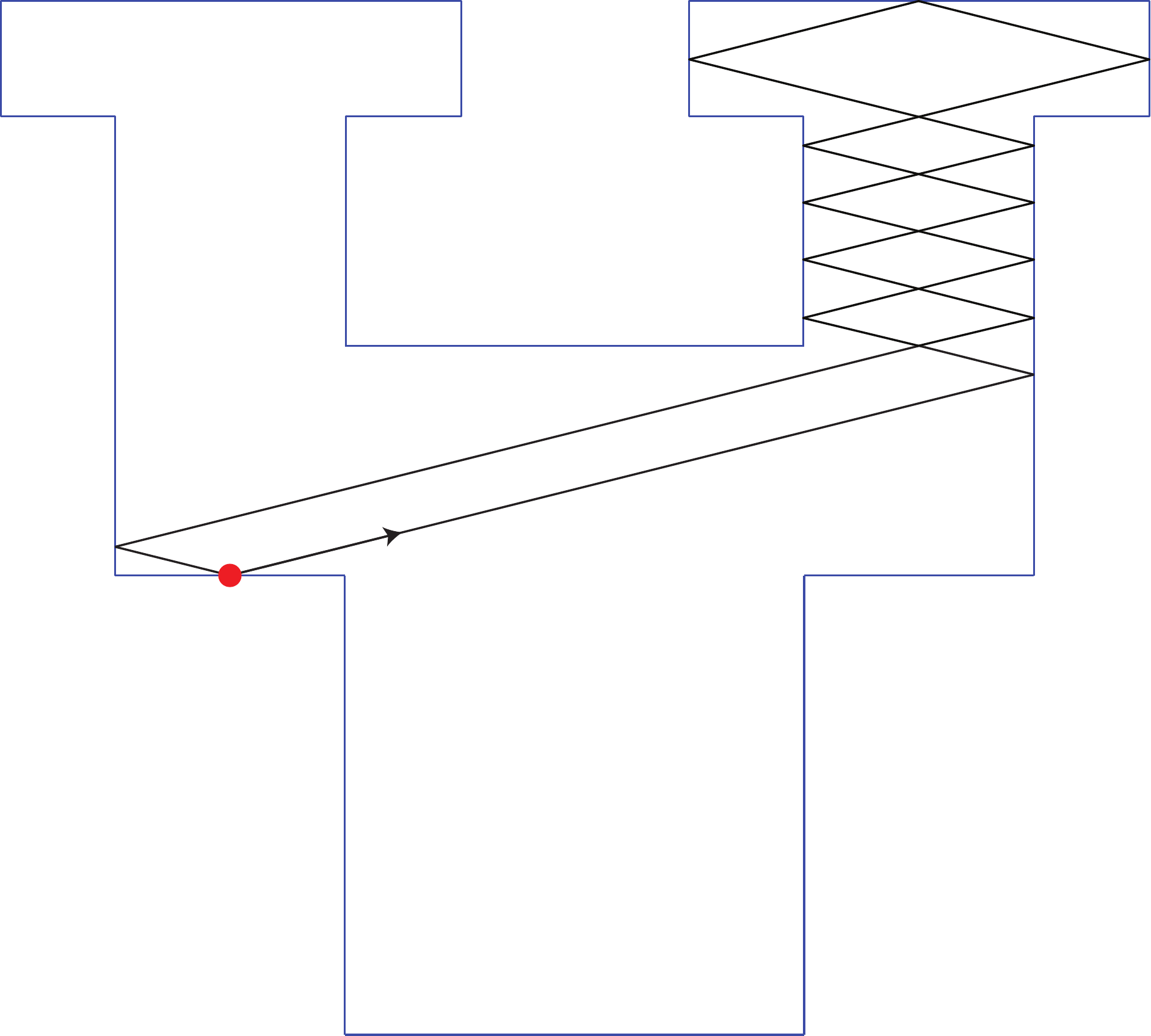}
\end{center}
\caption{Shown here is an orbit of an eventually constant sequence of compatible periodic orbits $\seqi{0}$.  More importantly, for every $n\geq 1$, we have that $\orbitiang{n}{\theta^0} = \orbitiang{1}{\theta^0}$.}
\label{fig:ConstantOrbitInT1}
\end{figure}

\section{Limits of particular sequences of compatible orbits}

\label{sec:NontrivialPathsInTheTFractalBilliardTable}

In \S\ref{sec:SequencesOfCompatiblePeriodicOrbits}, we showed that there are sequences of compatible periodic orbits that are eventually constant; see Theorem \ref{thm:eventuallyConstantSequenceOfCompatOrbits} and Figure \ref{fig:ConstantOrbitInT1}. In this section, we will focus on determining nontrivial paths and periodic orbits of the $T$-fractal billiard.  Specifically, we will be focusing on determining periodic orbits and nontrivial paths from sequences of compatible orbits with certain initial conditions (recall that a sequence of compatible orbits is determined by the initial condition of the first orbit in the sequence of compatible orbits).  Such initial conditions will be of the form $\xoo \neq m2^{-l}$, with $m$ and $l$ being positive integers and $\theta^0$ such that $\tan\theta^0 = \pm \frac{1}{p}$, where $p$ is an odd positive integer.  Then, the initial condition $(\xoo,\theta^0)$ satisfies Theorem \ref{thm:sequenceOfCompatibleNonsingularOrbitsPeriodicIfRational}, meaning that $\seqi{0}$ is a sequence of compatible periodic orbits.  What we will now see is that $\seqi{0}$ is not an eventually constant sequence of compatible periodic orbits.  More to the point, we will see that for every $n\geq 0$, $\tau_n <\infty$, $\tau_n<\upsilon_n$ and $\orbitixang{n}{\xio{n}}{\theta^0}_{\tau_n} = \ntpixang{n}{\xio{n}}{\theta^0}$ will be part of a nontrivial path of $\omegat$, where such a nontrivial path reaches a rational elusive point if and only if $\xoo$ is a rational value.  In \S\ref{sec:aNontrivialPathInAnIrrationalDirection}, we provide an example of a sequence of compatible singular orbits that yields a nontrivial path.  We begin with a specific example of the former as a motivation for some of the following results.


\begin{example}
\label{exa:nontrivialPaths1-3}
Consider $\xio{0} =  \frac{1}{3}$, along the base of $\omegati{0}$, and $\theta^0$ such that $\tan{\theta^0} = \frac{1}{3}$. Then, there exists a nonnegative integer $n$ such that $\tau_n<\infty$, $\tau_n<\upsilon_n$ and $\xii{n}{\tau_n}$ is a distance of $(3\cdot 2^{n})^{-1}$ away from a corner and the direction of the orbit prior to collision is $\theta^0$.  That is, heuristically speaking, there is a way for the billiard ball to reach ``1/3'' on a scaled segment.   Consequently, the path  $\ntpixang{n}{\frac{1}{3}}{\theta^0}$ derived from $\orbitixang{n}{\frac{1}{3}}{\theta^0}_{\tau_n}$ can be scaled by $2^{-n}$, and appended to $\xii{n}{\tau_n}$ in order to produce the next $\tau_n$-many segments in the path derived from the orbit $\orbitixang{2n}{\frac{1}{3}}{\theta^0}$ of $\omegati{2n}$.  Continuing in this fashion ad infinitum, we can show that such a procedure produces a nontrivial path of $\omegat$; see the image on the left in Figure \ref{fig:T-fractalNontrivialPaths}.  A similar construction produces a nontrivial path in the direction $\pi-\theta^0$ starting from $\frac{1}{3}$, as shown on the right in Figure \ref{fig:T-fractalNontrivialPaths}.
\end{example}

\begin{figure}
\begin{center}
\includegraphics[scale=.45]{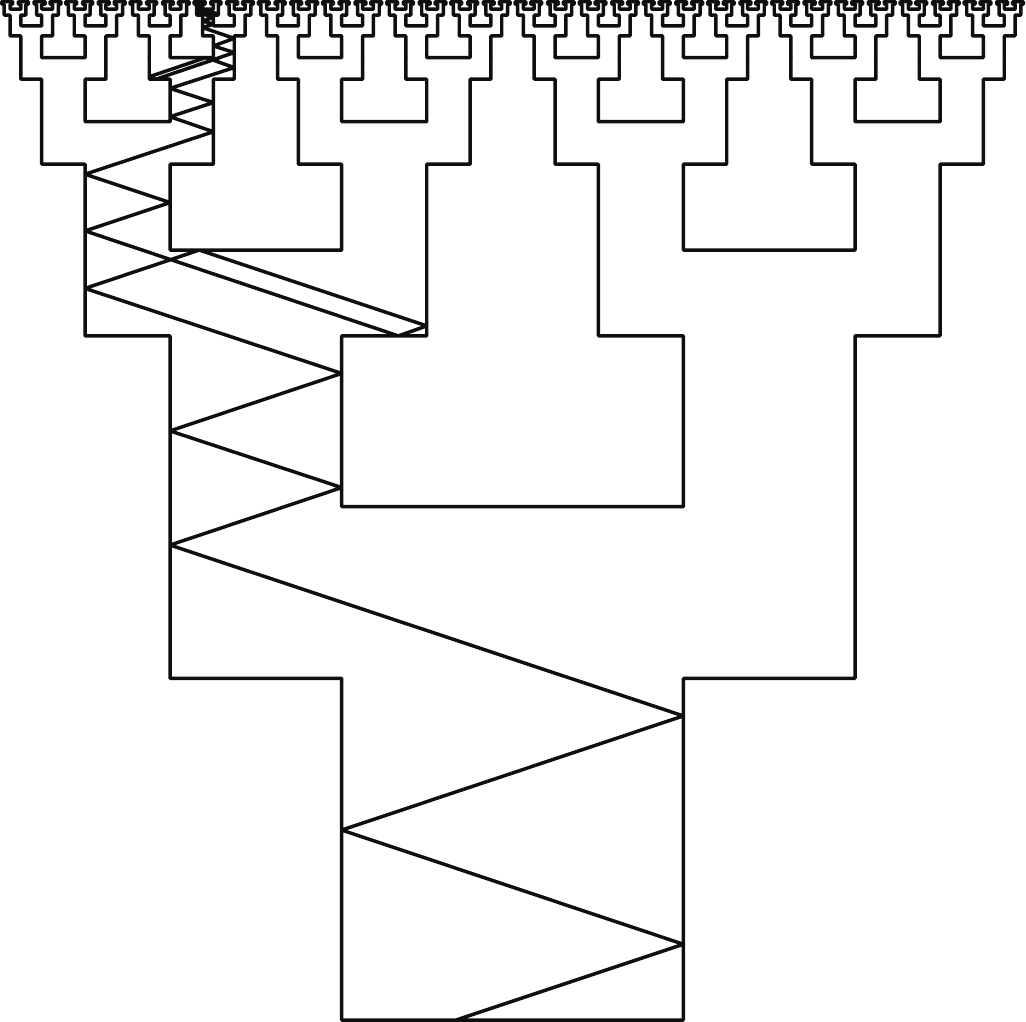}
\includegraphics[scale=.45]{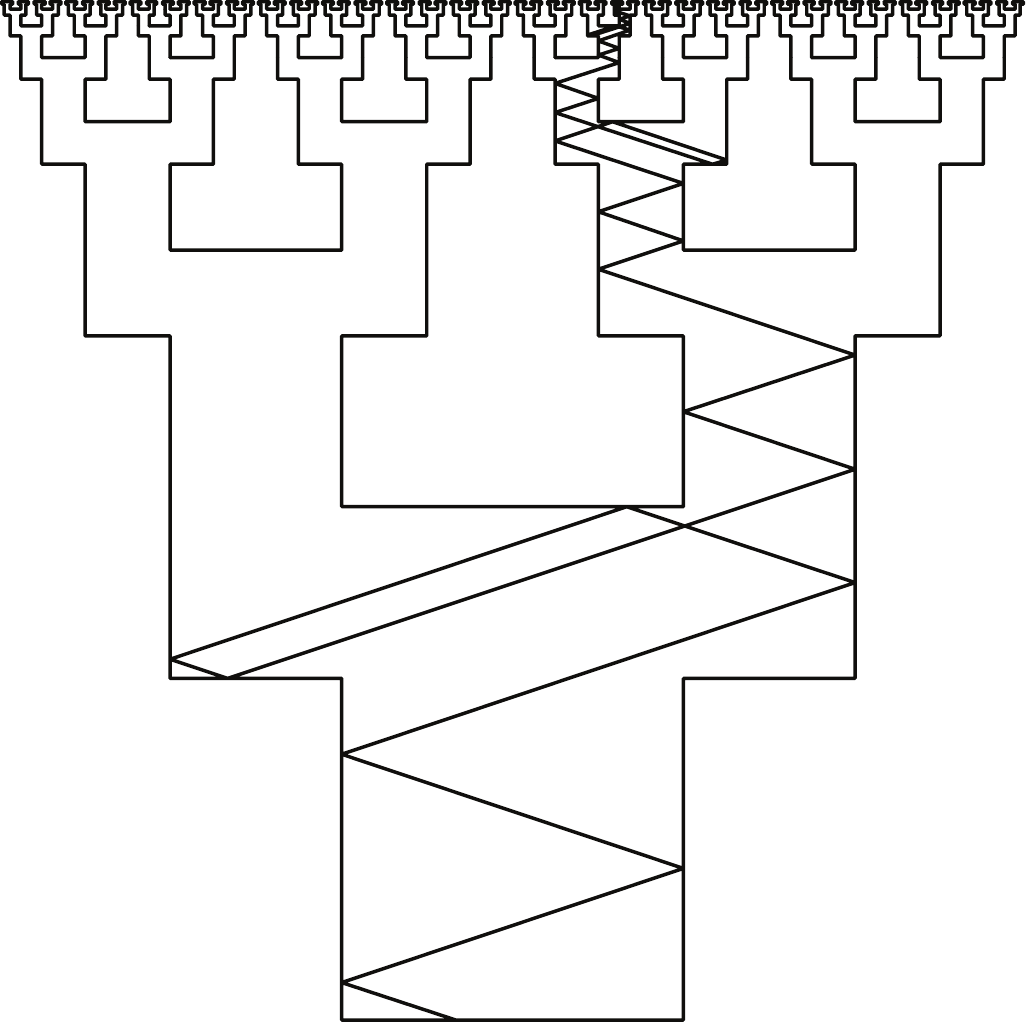}
\end{center}
\caption{Two nontrivial paths reaching two distinct elusive points of $\omegat$. As discussed in Example \ref{exa:nontrivialPaths1-3}, each nontrivial path can be constructed by way of repeatedly scaling finitely many segments and appending to such segments the scaled copy.}
\label{fig:T-fractalNontrivialPaths}
\end{figure}

\begin{lemma}
\label{lem:orbitsInStump}
Consider a unit-square billiard table $\Omega(S)$.  Let $p\geq 1$ be an odd integer and let $x^0 \neq m2^{-l}$, for any positive integers $l,m$, be a point on the base of $\Omega(S)$.    Let $\theta^0\in (0,\pi)$ be such that $\tan\theta^0=\frac{1}{p}$ \emph{(}resp., $\tan \theta^0 = -\frac{1}{p}$\emph{)}.  Then, any orbit $\mathscr{O}_S(x^0,\theta^0)$ with $x^0\in (0,1/2)$ on the base of $S$ necessarily intersects the boundary $S$ at $(x^{p+1},1)\in (1/2,1)\times\{1\}$. Similarly, if $x^0\in (1/2,1)$ on the base of $S$, then the orbit $\mathscr{O}_S(x^0,\theta^0)$ intersects the boundary $S$ at $(x^{p+1},1)\in (0,1/2)\times\{1\}$.  In the first \emph{(}resp., second\emph{)} case, the direction in which the pointmass is moving is $\pi - \arctan 1/p$ \emph{(}resp., $\arctan 1/p$\emph{)}.  Moreover, in either case, $x^{p+1}$ is a horizontal distance away from the left-hand side of $\Omega(S)$ given by $1-x^0$.
\end{lemma}

\begin{example}
Illustrated in Figure \ref{fig:illustratingOrbitHittingTopOfSquare} is an orbit of the square billiard table that begins at $x^0 = 1/3$ in an initial direction of $\theta^0$ such that $\tan\theta^0 = 1/3$.  Such an orbit exhibits the behavior described in Lemma \ref{lem:orbitsInStump}.
\end{example}

\begin{figure}
	\begin{center}
		\includegraphics{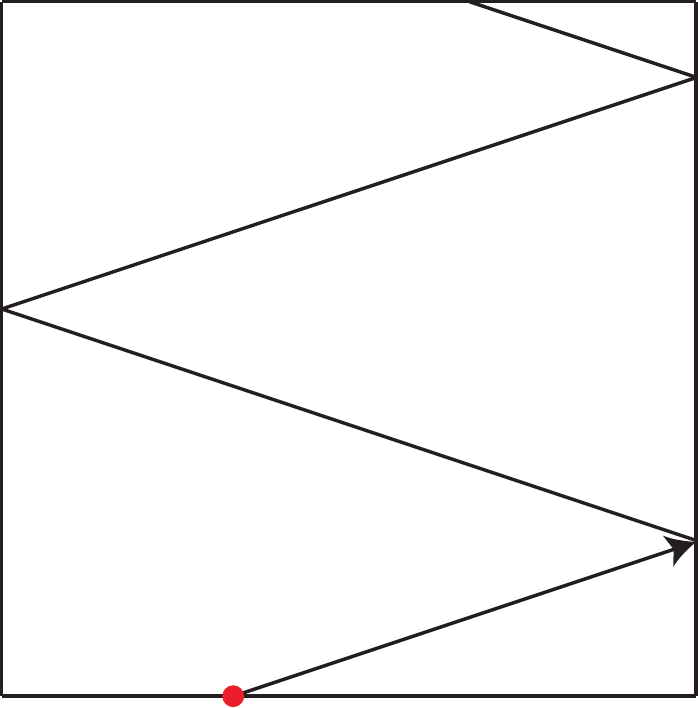}
	\end{center}
	\caption{An orbit of the square billiard table.  This figure illustrates an example of an orbit described in Lemma \ref{lem:orbitsInStump}.  The initial condition of the orbit is $x^0 = \frac{1}{3}$ and $\theta^0 = \arctan 1/3$.}
	\label{fig:illustratingOrbitHittingTopOfSquare}
\end{figure}



\begin{lemma}
\label{lem:orbitsInTop}
Let $p\geq 1$ be an odd integer. Consider a rectangular billiard table $\Omega(R)$ measuring $4$ units by $1$ unit.  Let $\theta^0\in (0,\pi)$ be such that $\tan(\theta^0)=\frac{1}{p}$ and $x^0 \neq 0,1,2,3,4$  lying on the base of $R$.  Then:

\begin{itemize}
\item{An orbit $\mathscr{O}(x^0,\theta^0)$ with $1<x^0<2$}
\begin{itemize}
\item{intersects the segment $\overline{(3,1),(4,1)}$ at a point $x^k$ if $p\equiv 1 \mod 8$ or $p\equiv 3\mod 8$ at an angle $\theta^0$;}
\item{intersects the segment $\overline{(0,1),(1,1)}$ at a point $x^k$ if $p\equiv 5 \mod 8$ or $p\equiv 7 \mod 8$ at an angle $\theta^0$. }
\end{itemize}

\item{An orbit $\mathscr{O}(x^0,\theta^0)$ with $2<x^0<3$}
\begin{itemize}
\item{intersects the segment $\overline{(3,1),(4,1)}$ at a point $x^k$ if $p\equiv 1 \mod 8$ or $p\equiv 3\mod 8$  at an angle $\pi - \theta^0$;}
\item{intersects the segment $\overline{(0,1),(1,1)}$ at a point $x^k$ if $p\equiv 5\mod 8$ or $p\equiv 7 \mod 8$ at an angle $\pi - \theta^0$.}
\end{itemize}
\end{itemize}

\noindent Similarly, let $\theta^0\in (0,\pi)$ be such that $\tan(\theta^0)=-\frac{1}{p}$, with $p=2k+1$ for some nonnegative integer $k$ and $x^0 \neq 0,1,2,3,4$  lying on the base of $R$.  Then:
\begin{itemize}
\item{An orbit $\mathscr{O}(x^0,\theta^0)$ with $1<x^0<2$}
\begin{itemize}
\item{intersects the segment $\overline{(3,1),(4,1)}$ at a point $x^k$ if $p\equiv 5 \mod 8$ or $p\equiv 7\mod 8$ at an angle $\theta^0$;}
\item{intersects the segment $\overline{(0,1),(1,1)}$ at a point $x^k$ if $p\equiv 1 \mod 8$ or $p\equiv 3 \mod 8$ at an angle $\theta^0$.}
\end{itemize}

\item{An orbit $\mathscr{O}(x^0,\theta^0)$ with $2<x^0<3$}
\begin{itemize}
\item{intersects the segment $\overline{(3,1),(4,1)}$ at a point $x^k$ if  $p\equiv 5 \mod 8$ or $p\equiv 7 \mod 8$  at an angle $\pi - \theta^0$;}
\item{intersects the segment $\overline{(0,1),(1,1)}$ at a point $x^k$ if $p\equiv 1\mod 8$ or $p\equiv 3 \mod 8$ at an angle $\pi - \theta^0$.}
\end{itemize}
\end{itemize}
\end{lemma}

\begin{proof}
This follows by inspection in each case and by recognizing the fact that the translation surface $\mathscr{S}(R)$ is tiled by $8$ unit squares in the horizontal direction; see \cite{MasTa} for a detailed discussion of translation surfaces and \cite{LapNie3,LapNie4} for a brief introduction. 
\end{proof}

Consider $\xoo\in I=[0,1]$, $\xoo\neq 0,1$.  Partitioning $I$ into $[0,1/2)$ and $[1/2,1]$, and supposing $\xoo\neq m2^{-l}$ for any positive integers $l,m$, then $\xoo$ is either in $(0,1/2)$ or $(1/2,1)$.  Suppose we write $\xoo$ in terms of its binary expansion (which is an infinite binary expansion, since this is not a dyadic rational).  If $\xoo\in (0,1/2)$, then rescaling $\xoo$ by $2$ results in $\xoo$ shifted to the left by one digit.  For example, $1/3=0.\overline{01}$ scaled by $2$  is $2/3$, which has the binary representation $0.\overline{10}$.  If $\xoo\in (1/2,1)$, then $(2\xoo)\mod 1$ is equal to the mantissa of the shift of $\xoo$ to the left by one digit.  For example, $5/9 = 0.1\overline{000111}$ and $(10/9\mod 1) = 1/9$, which has a binary representation $0.\overline{000111}$.

Now consider $\xoo\in I$ and $\theta^0 = (0,\pi/2)$.  Then, we say that the origin, relative to the direction $\theta^0$, is $(0,0)$.    More succinctly, $(0,0)$ is the \textit{relative origin} of $\xoo$.    If, on the other hand $\xoo\in I$ and $\theta^0 = (\pi/2,\pi)$, then we say that the relative origin of $\xoo$ is $(1,0)$.  Specifying a basepoint $\xoo\neq m2^{-l}$ on the base of $\omegati{0}$ and an angle, it is clear that we can always specify a relative origin in a well-defined manner.  Suppose $\seqiang{i}{\theta^0}$ is a sequence of compatible orbits such that, for every $n\geq 0$, there exists a finite first escape time $\tau_n$ (in the sense of Definition \ref{def:firstEscapeTime}). At each $\xii{n}{\tau_n}$, the direction of motion prior to collision is given by either $\theta^0$ or $\pi-\theta^0$, neither of which describes a vertical direction of flow. Hence, at each $\xii{n}{\tau_n}$, a relative origin can be described in a well-defined manner; see Figure \ref{fig:relativeOrigin}.

We introduce some related notation.  Suppose $\xoo$ is a point on the base of $\omegati{0}$.  Then $\xoo$ has a binary expansion and we represent the first $k$-many digits of this expansion by $(\xoo)_k$, $k\geq 1$.  In the sequel, we will relate the distance $\xii{n}{\tau_n}$ is from its relative origin to the the finite binary expansion $(\xoo)_{n+1}$.  Moreover, as we indicated above, $\xoo$ may be represented as either a rational value or as a binary expansion.  It will be clear from the context which representation of $\xoo$ we are using.  For example, when we compute $\xoo - (\xoo)_n$, we are supposing that $\xoo$ is written as an infinite binary expansion.  Additionally, the notation $(\xoo)_{k,k+1}$ represents the $(k+1)$-th digit in the binary expansion of $\xoo$ preceded by $k$ zeros.

\begin{figure}
\begin{center}
\includegraphics[scale=.36]{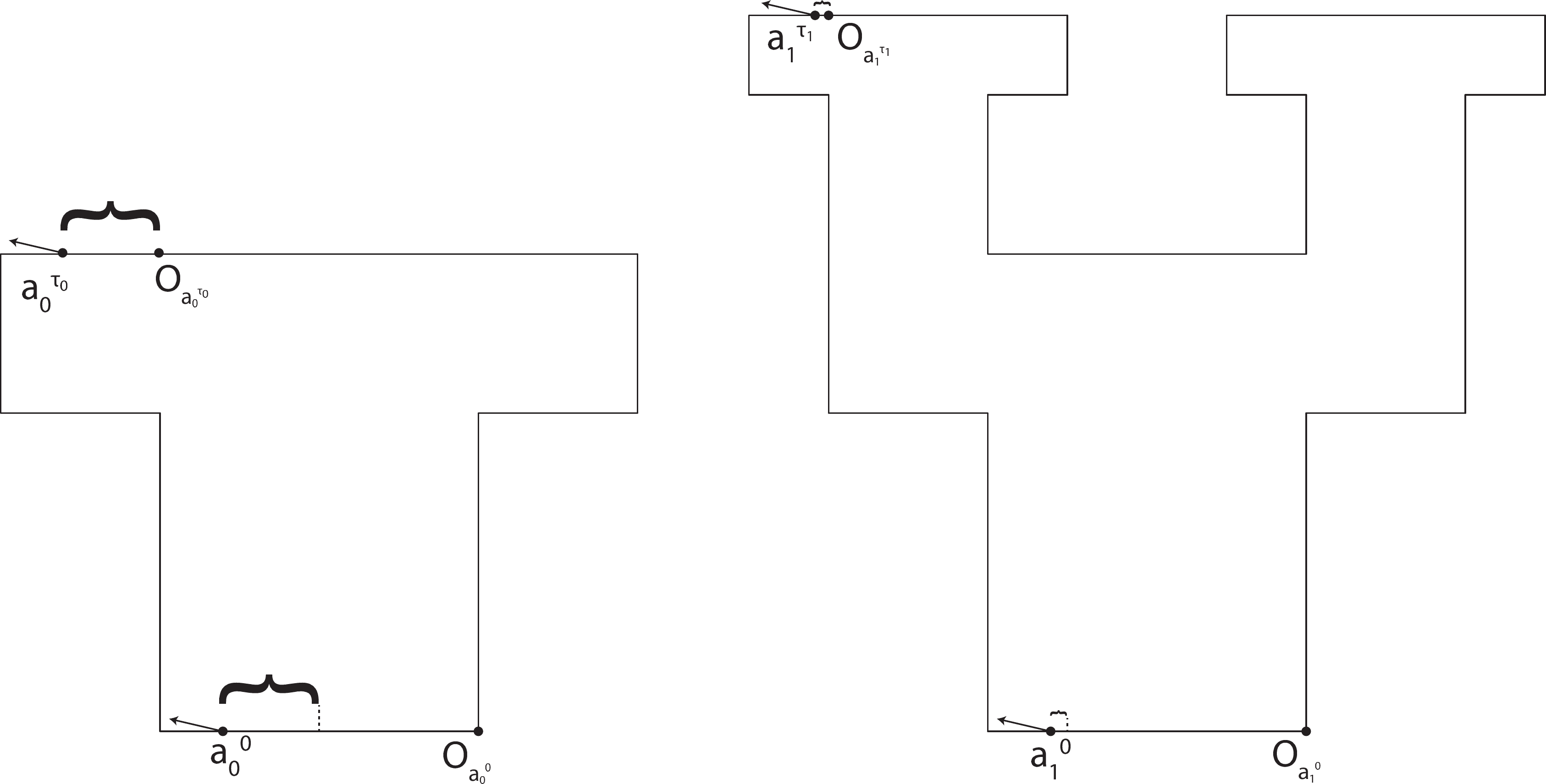}
\end{center}
\caption{In the figure on the left, we see the relative origin of the initial basepoint $a_0^0$, denoted by $O_{a_0^0}$.  The relative origin of $a_0^{\tau_0}$ is $O_{a_0^{\tau_0}}$; it is given on the right.  We see that for $i=0,1$, $a_i^{\tau_i}$ is a distance away from its relative origin given by $(1-\xoo)-(1-\xoo)_{i+1}$.} 
\label{fig:relativeOrigin}
\end{figure}

Lemmas \ref{lem:orbitsInStump} and \ref{lem:orbitsInTop} imply part of the following lemma.

\begin{lemma}
\label{lem:stumpAndTopCombined}
Let $p\geq 1$ be an odd integer.  Consider $\omegati{0}$ and an orbit $\orbiti{0}$ such that $\xoo$ is on the base of $\omegati{0}$, $\xio{0} \neq m2^{-l}$ for any positive integers $m,l$ and $\theta^0_0$ such that $\tan\theta_0^0 = \pm \frac{1}{p}$.  Then, the first escape time $\tau_0$ of $\orbiti{0}$ is finite and less than the first return time $\upsilon_0$.  

Moreover, if $\tan\theta^0_0 = \frac{1}{p}$, then $\xii{0}{\tau_0}$ is a distance of $\xoo - (\xoo)_{1}$ away from its relative origin $O_{\xii{0}{\tau_0}}$.  If $\tan\theta^0_0 = -\frac{1}{p}$, then $\xii{0}{\tau_0}$ is a distance of $(1 - \xoo)-(1-\xoo)_{1}$ away from its relative origin $O_{\xii{0}{\tau_0}}$. 
\end{lemma}

\begin{proof}
Consider $\omegati{0}$ and an orbit $\orbiti{0}$ such that $\xoo$ and $\theta^0_0$ are as described in the hypotheses above.  The fact that the first escape time $\tau_0$ is finite and less than the first return time $\upsilon_0$ of the orbit $\orbiti{0}$ follows immediately from Lemmas \ref{lem:orbitsInStump} and \ref{lem:orbitsInTop}.  

Suppose  that $\tan\theta^0_0 = 1/p$.  To see why $\xii{0}{\tau_0}$ is a distance of $\xoo - (\xoo)_1$ away from its relative origin $O_{\xii{0}{\tau_0}}$, we proceed as follows.   Either $\xoo < 1/2$ or $\xoo>1/2$.  

If $0<\xoo <1/2$, then
\begin{align}
\notag \xoo = \xoo - 0.0& = \xoo - (\xoo)_1.
\end{align}
\noindent When the billiard ball reaches the top edge of $\omegati{0}$, it will be located at a point $\xii{0}{\tau_0}$ that is a distance of $\xoo$ away from the relative origin $O_{\xii{0}{\tau_0}}$.  This follows from the fact that when one is examining the billiard orbit in $\omegati{0}$, one may do so as if it were the reflected-unfolded orbit of a square billiard table with side-length measuring $1/2$.  Hence, a billiard ball beginning a certain distance away from a corner of a square will intersect the top of the square at the same distance from some other top corner.  This corner would then be the relative origin of this basepoint in the smaller square billiard table. Since one is reflecting-unfolding an orbit into the approximation $\omegati{0}$, the billiard ball will intersect a segment to be removed in the construction of $\omegati{1}$ from $\omegati{0}$, but will be located a distance of $\xoo$ from its relative origin $O_{\xii{0}{\tau_0}}$.

On the other hand, if $1/2<\xoo<1$, then
\begin{align}
\notag\xoo - (\xoo)_1 &= \xoo - 0.1. 
\end{align}

\noindent When the billiard ball reaches the top edge of $\omegati{0}$, it will be located at a point $\xii{0}{\tau_0}$ that is a distance of $\xoo - 0.1$ away from the relative origin $O_{\xii{0}{\tau_0}}$.  The argument supporting this assertion is similar to the one given in the preceding paragraph.

When $\tan\theta^0_0 = -1/p$, an entirely similar argument can be used in order to show that $\xii{0}{\tau_0}$ is a distance of $(1 - \xoo)-(1-\xoo)_{1}$ away from its relative origin $O_{\xii{0}{\tau_0}}$. 

\end{proof}

\begin{proposition} 
\label{prop:FirstReturnTimeGreaterThanFirstEscapeTime}
Let $p\geq 1$ be an odd integer and $\xio{0} \neq m2^{-l}$, for any positive integers $l,m$, be on the base of $\omegati{0}$.  If $\seqiang{0}{\theta^0}$ is a sequence of compatible periodic orbits with $\tan{\theta^0} = \frac{1}{p}$ then, for every $n\geq 0$, there exists a finite first escape time $\tau_n< \upsilon_n$ and $\xii{n}{\tau_n}$ lying on a segment of $\omegati{n}$ to be removed in the construction of $\omegati{n+1}$ from $\omegati{n}$.

Moreover, still if $\tan\theta^0 = \frac{1}{p}$, then $\xii{n}{\tau_n}$ is a distance of $\xoo - (\xoo)_{n+1}$ away from its relative origin $O_{\xii{n}{\tau_n}}$.  If $\tan\theta^0 = -\frac{1}{p}$, then $\xii{n}{\tau_n}$ is a distance of $(1 - \xoo)-(1-\xoo)_{n+1}$ away from its relative origin $O_{\xii{n}{\tau_n}}$. 

\end{proposition}

\begin{proof}
Let $\seqi{0}$ be a sequence of compatible periodic orbits and let $\xoo$ and $\theta^0$ be as described in the hypotheses above.  

 We proceed by induction on $n$.  The basic case when $n=0$ is stated and dealt with in Lemma \ref{lem:stumpAndTopCombined}. Let $N>0$.  For every $n\leq N$, we assume that $\xii{n}{\tau_n}$ does not have a finite binary expansion, determined relative to the segment on which $\xii{n}{\tau_n}$ lies, and that $|\xii{n}{\tau_n} - O_{\xii{n}{\tau_n}}| = \xoo - (\xoo)_{n+1}$.  Then, Lemma \ref{lem:stumpAndTopCombined} shows that the first escape time $\tau_{n+1}$ is finite and less than the first return time $\upsilon_{n+1}$.  

We then see that
\begin{align}
\notag \xoo - (\xoo)_{n+2} &= \xoo - ((\xoo)_{n+1}+(\xoo)_{n+1,n+2})\\
                                            &= \xoo - (\xoo)_{n+1}-(\xoo)_{n+1,n+2}\\
            \notag                         &= |\xii{n}{\tau_n}-O_{\xii{n}{\tau_n}}| -  (\xoo)_{n+1,n+2}.                    
\end{align}

The midpoint of the segment on which $\xii{n}{\tau_n}$ lies is between $\xii{n}{\tau_{n}}$ and $O_{\xii{n}{\tau_n}}$ if and only if $(\xoo)_{n+1,n+2} = 1$.  Therefore,

\begin{align}
\notag      \xoo - (\xoo)_{n+2} &= |\xii{n+1}{\tau_{n+1}}-O_{\xii{n+1}{\tau_{n+1}}}|.    
\end{align}

\noindent When $\tan\theta^0 = -1/p$, an entirely similar argument can be used in order to show that $\xii{n}{\tau_n}$ is a distance of $(1 - \xoo)-(1-\xoo)_{n+1}$ away from its relative origin $O_{\xii{n}{\tau_n}}$. 

This completes the proof of the proposition.  
\end{proof}

\begin{corollary}
\label{cor:twoNontrivialPaths}
Let $p\geq 1$ be an odd integer and $\xio{0} \neq m2^{-l}$, for any positive integers $l,m$, be on the base of $\omegati{0}$.  Suppose $\seqiang{0}{\theta^0}$ is a sequence of compatible periodic orbits with $\tan{\theta^0} = \frac{1}{p}$.  Then, in the sense of Definition \ref{def:ANontrivialPath}, both  $\seqiang{0}{\theta^0}$ and  $\seqiang{0}{\pi -\theta^0}$ yield nontrivial paths $\ntpang{\theta^0}$ and $\ntpang{\pi-\theta^0}$, respectively, where $x^0 := \xoo$. \emph{(}We remind the reader that $\xio{n} = \xoo$ for all $n\geq 0$.\emph{)}
\end{corollary}

\begin{proof}
Let $\seqi{0}$ be a sequence of compatible periodic orbits, with $\xoo$ and $\theta^0$ as described in the hypotheses above.  From the sequence of basepoints $\{\xii{n}{\tau_n}\}_{n=0}^\infty$, we can construct a sequence of paths $\{\ntpi{n}\}_{n=0}^\infty$ converging to a nontrivial path $\ntp$ that reaches an elusive point of $\omegat$.  Specifically, $\ntp = \overline{\bigcup_{n=0}^\infty\ntpi{n}}$, where the closure is with respect to the Hausdorff metric (in the plane, $\mathbb{R}^2$).\footnote{Recall that the Hausdorff metric is a metric on the space of nonempty compact subsets of a metric space $X$.  If $A$ and $B$ are two nonempty compact subsets of $X$, then $d_H(A,B) := \max\{\sup_{a\in A}\inf_{b\in B} d_X(a,b),\sup_{b\in B}\inf_{a\in A} d_X(a,b)\}$.}\footnote{Really, we should work with the Hausdorff--Gromov limit here; see footnote \ref{ftnote:HausdorffGromovLimit}.} Similarly, $\{\ntpixang{n}{\xio{n}}{\pi - \theta^0}\}_{n=0}^\infty$ reaches an elusive point.
\end{proof}

In Figure \ref{fig:T-fractalNontrivialPaths}, we see an example of two nontrivial paths described in Corollary \ref{cor:twoNontrivialPaths}.  Beginning from $x^0 = \frac{1}{3}$ in the direction $\theta^0$ such that $\tan\theta^0 = \frac{1}{3}$, $\ntp$ and $\ntpang{\pi - \theta^0}$ are two nontrivial paths each reaching distinct elusive points.  As we will now see, such elusive points are rational points since $x^0$ is a rational value.


\begin{theorem}
\label{thm:nontrivialPathsConvergeToRationalElusivePointsIffXooIsRational}
Let $p\geq 1$ be an odd integer and $\xio{0} \neq m2^{-l}$, for any two positive integers $l,m$.  Suppose $\seqiang{0}{\theta^0}$ is a sequence of compatible periodic orbits with $\tan{\theta^0} = \frac{1}{p}$. Then, $\ntpang{\theta^0}$ and $\ntpang{\pi-\theta^0}$ converge to two distinct elusive points, and these two elusive points are rational elusive points if and only if $\xoo$ is a rational value. 
\end{theorem}

\begin{proof}
Corollary \ref{cor:twoNontrivialPaths} states that each nontrivial path will converge to some elusive point. Let $\seqi{0}$ be a sequence of compatible periodic orbits with $\xoo$ and $\theta^0$ as described in the hypotheses above.  Suppose $\xoo$ is a rational value.  By Proposition \ref{prop:FirstReturnTimeGreaterThanFirstEscapeTime}, there exists $N$ such that 
\begin{align}
\notag 2^N|\xii{N}{\tau_N} - O_{\xii{N}{\tau_N}}| & = |\xoo - O_{\xoo}|.
\end{align}
\noindent If $\theta^{\tau_N - 1} = \theta^0$ (the reflected angle after colliding with the boundary at the point $\xii{N}{\tau_N - 1}$), then one can copy, scale and append $\orbitiang{N}{\theta^0}_{\tau_N}$ to $\xii{N}{\tau_N}$ in order to determine the path $\orbitiang{2N}{\theta^0}_{\tau_{2N}}$ of $\omegati{2N}$.  Continuing in this fashion ad infinitum, one produces a nontrivial path reaching a rational elusive point.  The portion of the orbit given by $\orbitiang{N}{\theta^0}_{\tau_N}$ establishes the pattern of $L$'s and $R$'s one can use to address the rational elusive point of $\omegat$.  

If $\theta^{\tau_N - 1} = \pi - \theta^0$, then $\xii{N}{\tau_N}$ must still be the same distance away from its relative origin.  This means that one can simply reflect $\orbitiang{n}{\theta^0}_{\tau_N}$ through the vertical line of symmetry for $\omegat$, then scale, copy and append the path $\orbitiang{N}{\tau_N}$ to the point $\xii{N}{\tau_N}$  in order to produce the path $\orbitiang{2N}{\theta^0}_{\tau_{2N}}$.  One then continues to append scaled copies of $\orbitiang{2N}{\theta^0}_{\tau_{2N}}$, ad infinitum.  The eventual result is a nontrivial path reaching a rational elusive point of $\omegat$.  

If $\xoo$ is irrational, then $|\xoo - (\xoo)_{N}|$ is never rational.  Hence, for any $n\geq 0$, there is no way to scale $\orbitiang{n}{\theta^0}_{\tau_n}$ so as to append to $\xii{n}{\tau_n}$ and produce the path $\orbitiang{2n}{\theta^0}_{\tau_{2n}}$.  If there were such a way, this would imply that the corresponding orbit of a square of scale $2^{-n-1}$ would unfold as a straight-line into a tiling of the plane by squares with side-length $2^{-n-1}$ and intersect a horizontal segment at a point with an $x$-coordinate that would be rational, which is impossible. Hence, the corresponding nontrivial path converges to an irrational elusive point of $\omegat$.

\end{proof}


\begin{remark}
We note that in the proof above, when the direction of motion at $\xii{N}{\tau_N -1}$ is not $\theta^0$ but is instead $\pi - \theta^0$, the point of first escape $\xii{N}{\tau_N}$ is no longer the same point as when the direction of travel was assumed to be $\theta^0$.  This follows from the fact that we are measuring how far $\xii{N}{\tau_N}$ is from its relative origin, and that this distance must remain the same.  The notation remains the same, while the location of the first escape point must be different in order to maintain the same distance from the appropriate relative origin (which depends solely on the direction of travel). 
\end{remark}
\begin{figure}
	\begin{center}\includegraphics[scale=.7]{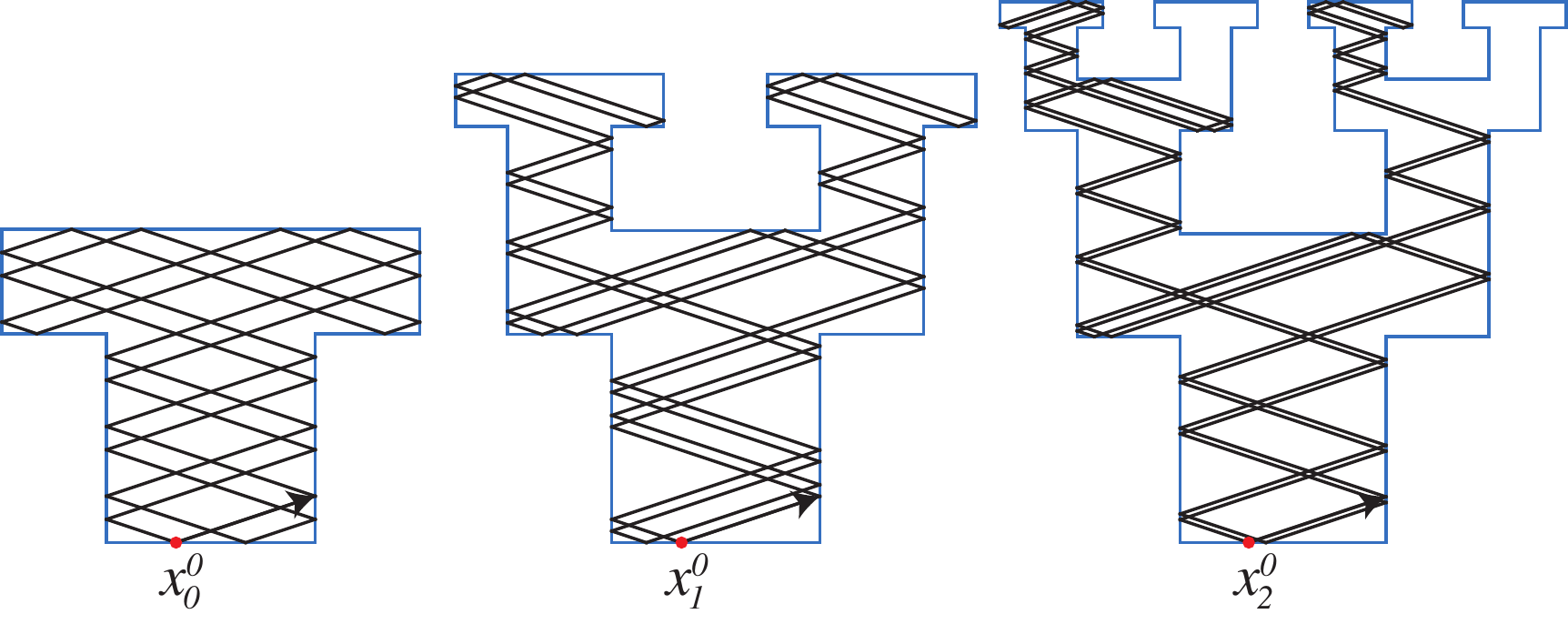}
	\end{center}
	\caption{Shown here is a sequence of compatible orbits satisfying Definitions \ref{def:recurrentOrbit} and \ref{def:periodicOrbit}; the initial condition of each orbit is $(\frac{1}{3},\arctan \frac{1}{3})$.  We can see how for all sufficiently large $n$, $\ntpixang{n}{\frac{1}{3}}{\arctan \frac{1}{3}}\cup\ntpixang{n}{\frac{1}{3}}{\pi - \arctan \frac{1}{3}} $ is virtually indistinguishable from $\orbitixang{n}{\frac{1}{3}}{\arctan\frac{1}{3}}$.  Hence, the orbit $\billiardOrbit$ is degenerate.}
	\label{fig:sequenceOfCompatibleOrbits1-3-1-3}
\end{figure}
In Theorem \ref{thm:nontrivialPathsConvergeToRationalElusivePointsIffXooIsRational}, we determined how a particular nontrivial path would converge to an elusive point.  An example of such a nontrivial path was given in Example \ref{exa:nontrivialPaths1-3} and illustrated in Figure \ref{fig:T-fractalNontrivialPaths}.  We now go one step further to show that the two nontrivial paths given in Theorem \ref{thm:nontrivialPathsConvergeToRationalElusivePointsIffXooIsRational} describe the path which the orbit $\billiardOrbit$ would traverse.  For the initial condition given in Theorem \ref{thm:PeriodicOrbitsFromNontrivialPaths} below, the orbit would be degenerate, meaning that the path traversed after reaching the elusive point is exactly the path traversed by the nontrivial path, but in the opposite direction.  This is illustrated in Figure \ref{fig:sequenceOfCompatibleOrbits1-3-1-3}.
\begin{theorem}[Periodic orbits from nontrivial paths]
\label{thm:PeriodicOrbitsFromNontrivialPaths}
Let $p\geq 1$ be an odd integer and $\xio{0} \neq m2^{-l}$, with $l,m$ being arbitrary positive integers.  Suppose $\seqiang{0}{\theta^0}$ is a sequence of compatible periodic orbits such that $\tan{\theta^0} = \frac{1}{p}$.  If $x^0 = \xoo$ \emph{(}where we recall that $\xio{n} = \xoo$ for all $n\geq 0$\emph{)}, then $\billiardOrbit$ is a periodic orbit of $\omegat$, in the sense of Definition \ref{def:periodicOrbit}.
\end{theorem}

\begin{proof}
Let $\seqi{0}$ be a sequence of compatible periodic orbits, with $\xoo$ and $\theta^0$ as described in the hypotheses above.  We claim that for each $n>0$, $\orbitiang{n}{\theta^0}$ is an orbit with a first escape time $\tau_n<\infty$ and first return time $\upsilon_n$ such that $\tau_n<\upsilon_n$ and $|\xii{n}{\upsilon_n}-\xoo| < 2^{-n+1}$.  The fact that $\tau_n$ is finite and $\tau_n<\upsilon_n$ follows from Corollary \ref{cor:twoNontrivialPaths}.  Using the symmetry of $\omegati{0}$ and Lemmas \ref{lem:orbitsInStump} and \ref{lem:orbitsInTop}, one can show that an orbit beginning from the base of $\omegati{0}$ at an angle of $\theta^0$ will return in an antiparallel direction $\theta^0+\pi$.  Consequently, in any scaled copy of $\omegati{0}$ in $\omegati{n}$, any orbit passing through a deleted segment at an angle $\theta^{k}$ will return and pass through the deleted segment at an angle of $\theta^k+\pi$ and less than a distance of $2^{-n+1}$ from the point at which the orbit initially crossed the deleted segment.  Hence, $\xii{n}{\upsilon_n}$ will be within a distance of $2^{-n+1}$ from $\xoo$. It follows that as $n$ increases, the distance between $\xii{n}{\upsilon_n}$ and $\xoo$ decreases to zero. Hence, $\{\xii{n}{\upsilon_n}\}_{n=0}^\infty$ will converge to some point $x$ in the unit-interval base of $\omegat$.  Specifically, $\xii{n}{\upsilon_n}$ converges to $\xoo:=x^0$ and the path that this orbit takes is given by the union of the two nontrivial paths shown to exists by Corollary \ref{cor:twoNontrivialPaths}.
\end{proof}

As was alluded to in \S\ref{sec:SequencesOfCompatiblePeriodicOrbits}, we now return to our discussion of eventually constant sequences of compatible periodic orbits.  Recall that the notion of an eventually constant sequence of compatible orbits was introduced in Definition \ref{def:EventuallyConstanceSequenceOfCompatibleOrbits} and that an eventually constant sequence of compatible periodic orbits is one for which every compatible orbit is periodic in its respective billiard table.  We now state the following result.

\begin{theorem}
\label{thm:eventuallyConstantSequenceIsPeriodicOrbit}
	The trivial Hausdorff limit of an eventually constant sequence of compatible periodic orbits is a periodic orbit of $\omegat$. 
\end{theorem}

\begin{proof}
In this case, Definition \ref{def:periodicOrbit} is clearly satisfied.	
\end{proof}

We can clearly see in Figure \ref{fig:ConstantOrbitInT1} why the Hausdorff limit of an eventually constant sequence of compatible periodic orbits $\seqi{i}$ is a periodic orbit of $\omegat$: there exists some $k\geq i$ such that $\orbitiang{n}{\theta^0} = \orbitiang{k}{\theta^0}$, for every $n\geq k$.

In \S\ref{sec:SequencesOfCompatiblePeriodicOrbits}, we showed that there exist sequences of compatible periodic orbits.  In some situations (Theorem \ref{thm:eventuallyConstantSequenceOfCompatOrbits}), such sequences were eventually constant.  In other situations (Proposition \ref{prop:FirstReturnTimeGreaterThanFirstEscapeTime}) such sequences were not eventually constant.  But, both Theorems \ref{thm:PeriodicOrbitsFromNontrivialPaths} and \ref{thm:eventuallyConstantSequenceIsPeriodicOrbit} showed that one could construct a periodic orbit from a sequence of compatible periodic orbits, though with both orbits behaving qualitatively different.

\section{A nontrivial path in an irrational direction}
\label{sec:aNontrivialPathInAnIrrationalDirection}

Finally, we provide an example of a sequence of compatible singular orbits yielding a nontrivial path that converges to a rational elusive point, yet has an initial direction which is irrational.   Hence, each orbit in the sequence of compatible orbits is only part of an orbit that is, in fact, singular in its respective approximation (each backwards orbit will be dense in its respective approximation).  Consequently, such a sequence of compatible orbits will yield a nontrivial path that constitutes a singular orbit of $\omegat$, per Definition \ref{def:ASingularOrbit}.

Consider the initial direction $\theta^0$ such that $\tan \theta^0 = \sqrt{2}/34$.   We claim that there exists $x^0$  in the base of $\omegat$ such that $\billiardOrbitang{\pi - \theta^0}$ is a singular orbit of $\omegat$ and the path given by the orbit is a nontrivial path,  $\ntpang{\pi - \theta^0}$, reaching a rational elusive point of $\omegat$. 

\begin{remark}
We note that $\theta^0 = \arctan \frac{\sqrt{2}}{34}$ is a specific direction.  The behavior we are about to describe was observed in our search for an orbit with an irrational initial direction ($\theta^0$ such that $\tan \theta^0$ is an irrational value) which was not dense in $\omegat$. Current work with C. Johnson is concerned with determining whether or not the corresponding backwards orbit is dense or singular, or neither.  We expect that the work of P. Hooper is relevant in this context and should allow us to conclude that there are irrational directions such that orbits in such directions are not uniformly distributed in $\omegat$; see \cite{JohNie2} and \cite{Hoo}.  The orbit $\billiardOrbit$, where $\theta^0 = \arctan \sqrt{2}/34$, is likely to be such an example, but some work remains to be done in order to rigorously establish this.
\end{remark}

In this particular setting, we denote by $\zeta_n$ the number of iterations of the billiard map required for the orbit to reach the bottom of $\omegati{n}$.  This notation is particularly advantageous  when the billiard ball is beginning from the top of $\omegati{n}$ (or any other segment of $\omegati{n}$ that is not the base of $\omegati{n}$).  Additionally, we denote by $\sigma_n$ the value $\sum_{i=0}^n 3\cdot 2^{-i-1}$, this being the height of $\omegati{n}$.

An orbit beginning at the point on the boundary of $\omegati{0}$ given by $(0,\sigma_0)$ at an angle of $2\pi - \theta^0$ must return to the base at the point with $x$-coordinate 
\begin{align}
\xii{0}{\zeta_0} :=& 51\sqrt{2}/2 - 36;
\end{align}
\noindent   see Figure \ref{fig:firstSingularOrbit}.  Suppose an orbit begins at $(0,\sigma_1)$ at an angle of $\pi + \theta^0$.  Then, as we have observed in our exact arithmetic simulations\footnote{\label{ftnote:CASComment} We have performed exact arithmetic computations in Matlab using the built-in computer algebra system MuPad.  This means that the computer does not perform any approximations of irrational values.  Rather, for example, $\sqrt{2}$ is exactly that and not represented by some rational approximation.}, the orbit passes through the point $(-\xii{0}{\zeta_0}/2, \sigma_0)$ and eventually, after finitely many more collisions, and before escaping to $\omegati{2}$, the orbit intersects the base at the point $(1649\sqrt{2})/4 - 583+\xii{0}{\zeta_0}$.  We define $s_0$ to be the value $1649\sqrt{2}/4 - 583$; see Figures \ref{fig:firstSingularOrbit} and \ref{fig:secondSingularOrbit}.

\begin{figure}
\begin{center}
\includegraphics[scale=0.5]{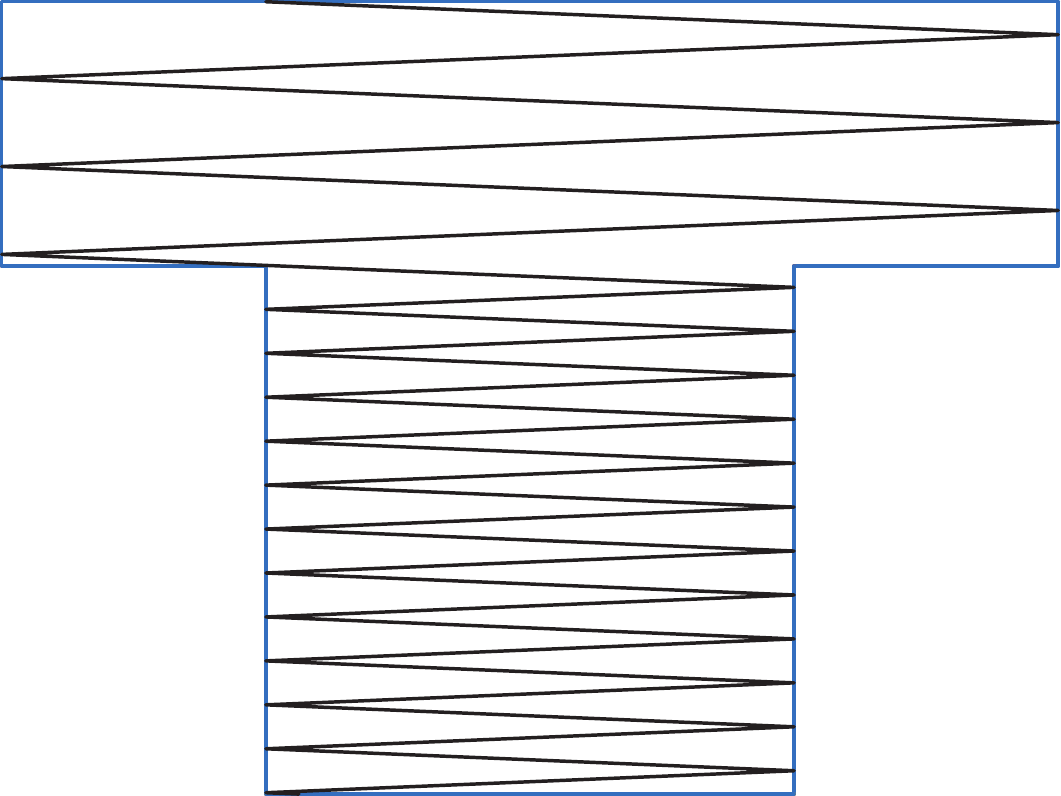}
\end{center}
\caption{Beginning at $(0,\sigma_0)$ in the direction $2\pi - \theta^0$, the billiard ball intersects the bottom of $\omegati{0}$ at the point $\xii{0}{\zeta_0}$.}
\label{fig:firstSingularOrbit}
\end{figure}

\begin{figure}
\begin{center}
\includegraphics[scale=0.5]{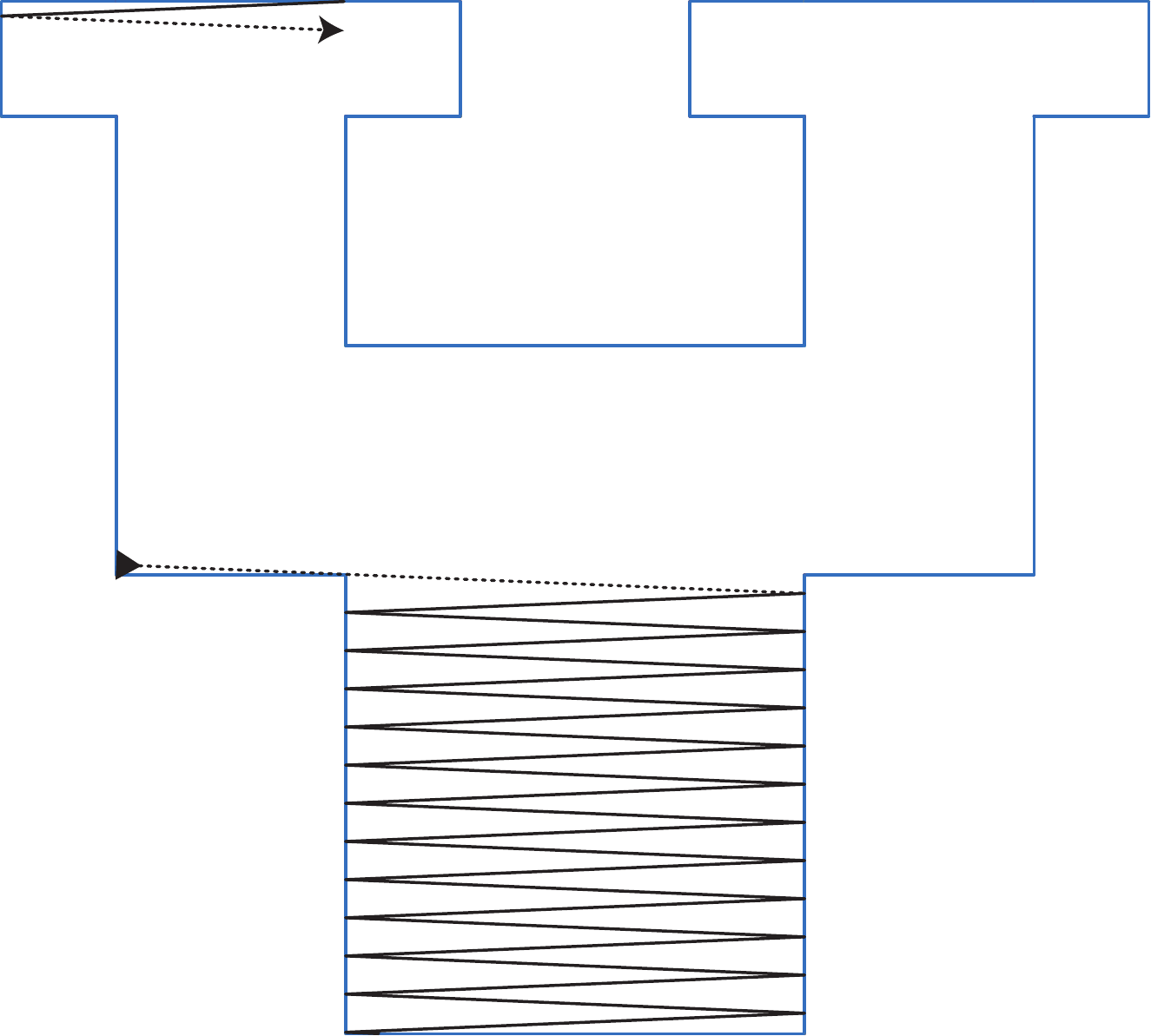}
\end{center}
\caption{Beginning at $(0,\sigma_1)$ in the direction $\pi + \theta^0$, the billiard ball intersects the bottom of $\omegati{1}$ at the point $\xii{1}{\zeta_1}$.}
\label{fig:secondSingularOrbit}
\end{figure}

\begin{figure}
\begin{center}
\includegraphics[scale=.85]{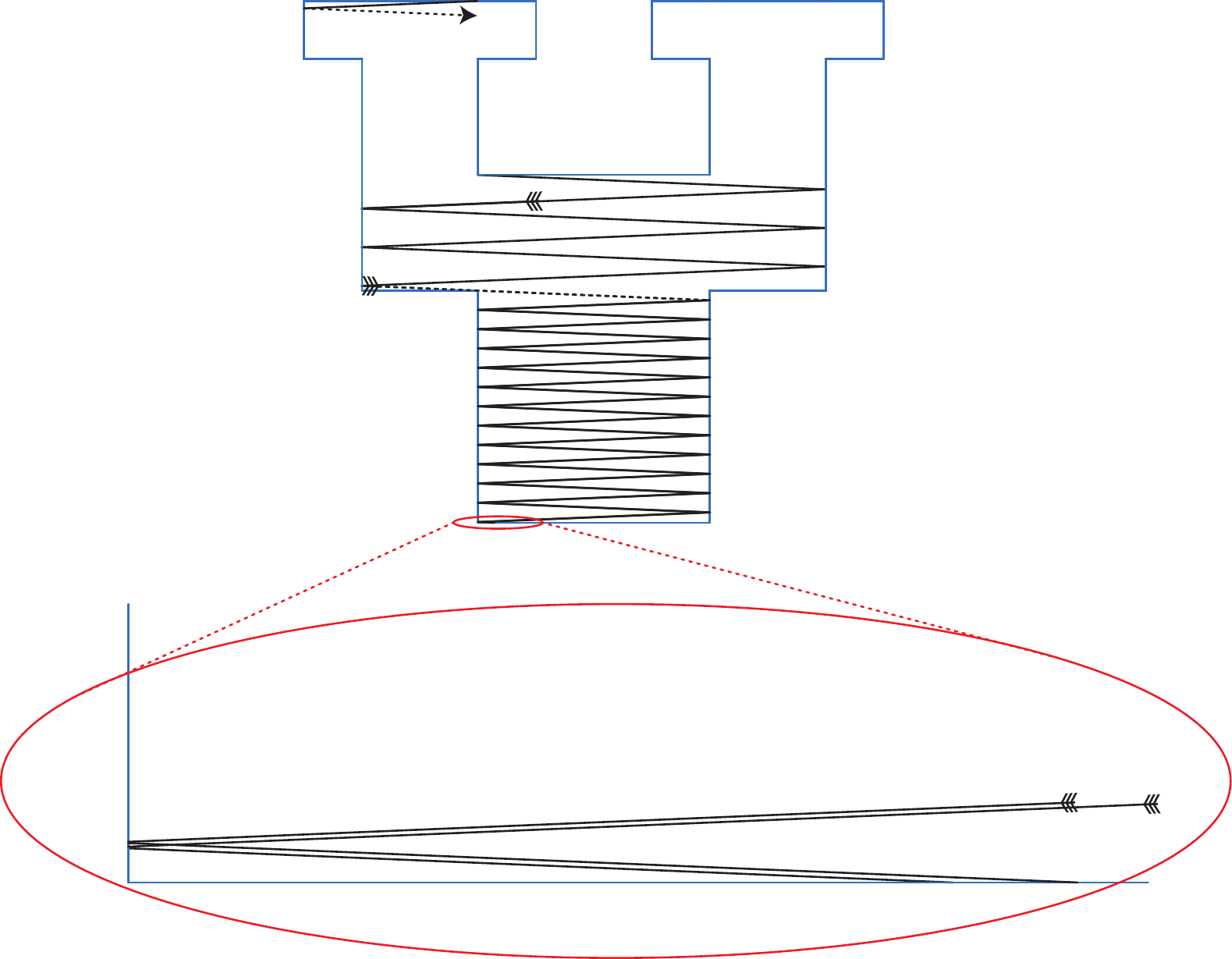}
\end{center}
\caption{Here, we have embedded into $\omegati{1}$ the orbit shown in Figure \ref{fig:firstSingularOrbit}. Moreover, $\xii{0}{\zeta_0}$ is shown to the left of $\xii{1}{\zeta_1}$, as is expressed in Equation (\ref{eqn:xnzetan}).  The dotted segment beginning at the top of $\omegati{1}$ and going in the direction of $\pi +\theta^0$  continues at the other dotted line, which is very close to the path given by the embedded orbit.  The full orbit was not shown, because it would obscure the important aspects of the two orbits.  The image on the bottom is meant to illustrate just how close together $\xii{0}{\zeta_1}$ and $\xii{1}{\zeta_1}$ are, and, consequently, how close together the paths given by the orbits are.}
\label{fig:firstAndSecondSingularOrbitsSuperimposed}
\end{figure}

Consider a sequence of initial conditions $\{(\xio{n},\theta^0_n)\}_{n=0}^\infty$, where $\xio{n} = (0,\sigma_n)$ (the top of $\omegati{n}$ with $x$-coordinate $0$) and
\begin{align}
\theta^0_n &= \left\{\begin{array}{ll}
2\pi - \theta^0,  & n = 2k \\
\pi + \theta^0, & n = 2k+1
\end{array}\right.
\end{align}
\noindent for all nonnegative integers $k$.  Observe that this is \textit{not a sequence of compatible} initial conditions.  

We assume that the basepoint  $\xii{n}{\zeta_n}$ of the orbit $\orbitix{n}{\xio{n}}$ lying on the base of $\omegati{n}$ is given by
\begin{align}
\label{eqn:xnzetan} \xii{n}{\zeta_n} &=\xii{0}{\zeta_0}+ \sum_{i=0}^n s_02^{-n}
\end{align}
\noindent for every $n\geq 1$.  We see that this assumption coincides with our exact arithmetic simulations\footnote{Recall that these computations are being performed using the computer algebra system built into Matlab; see the text of footnote \ref{ftnote:CASComment} for further elaboration.} of the following orbits: 
\begin{align}
\notag &\orbitixang{1}{(0,\sigma_1)}{\pi + \theta^0}\\
\notag &\orbitixang{2}{(0,\sigma_2)}{2\pi - \theta^0}\\
\notag & ...\\
\notag & \orbitixang{5}{(0,\sigma_5)}{\pi + \theta^0}.
\end{align}
 Moreover, 
\begin{align}
\notag \lim_{n\to\infty}\xii{n}{\zeta_n} &= \xii{0}{\zeta_0}+ \sum_{i=0}^\infty s_02^{-i}\\
                                                    &=\xii{0}{\zeta_0}+ 2s_0.
\end{align}
We define $y^0 := \xii{0}{\zeta_0}+2s_0$.

We now consider the sequence of compatible orbits $\seqixang{0}{y^0}{\pi -\theta^0}$.\footnote{In what follows, the notation $\orbiti{k}$ will be used interchangeably to refer to the orbit in its native approximation and also, to its embedding into subsequent approximations $\omegati{n}$, $n>k$, of the T-fractal billiard table.}  It can be shown through symbolic computation that 1) each (forward) orbit is singular and 2) a particular subset of the path of the orbit $\orbitixang{1}{y^0}{\pi -\theta^0}$ can be scaled by $2^{-1}$, reflected and translated so as to recover the part of $\orbitixang{2}{y^0}{\pi - \theta^0}$ missing from $\orbitixang{1}{y^0}{\pi -\theta^0}_{\tau_1}$, when embedded in $\omegati{2}$.  Since the orbit $\orbitixang{1}{y^0}{\pi - \theta^0}$ eventually intersects a corner of $\omegati{1}$ for some $k_n>\tau_n$, the following detailed explanation of how one scales and appends part of an orbit in order to determine the successive orbits can be used to justify why each orbit is in fact a singular orbit of its respective approximation. 

Let $\xio{1} = \xii{0}{\tau_0}$, where $\xii{0}{\tau_0}$ is the point at which $\orbitixang{0}{y^0}{\pi -\theta^0}$ escapes to $\omegati{1}$.  Consider the orbit $\orbitiang{1}{\pi-\theta^0}$.  Then, via our exact arithmetic simulations, we see that $\xii{1}{\tau_1} = -\xii{0}{\tau_0}/2$ and $\theta^{\tau_1} = \theta^0$.  Hence, one can continue copying, reflecting and appending the said portion of the orbit $\orbitixang{1}{y^0}{\pi -\theta^0}$ as a means for determining the orbit $\orbitixang{n}{y^0}{\pi -\theta^0}$, for every $n\geq 2$; see Figure \ref{fig:explainingIrrationalOrbitConstruction}.  

Moreover, each orbit $\orbitixang{n}{y^0}{\pi -\theta^0}$ intersects a corner of $\omegati{n}$ after intersecting with a segment that would be removed in the construction of $\omegati{n+1}$.  While such orbits would be dense in backward time, we are only interested here in determining a nontrivial path.  Hence, $\seqixang{0}{y^0}{\pi - \theta^0}$ is a sequence of compatible singular orbits yielding a nontrivial path.  This nontrivial path then constitutes a singular orbit of $\omegat$, in the sense of Definition \ref{def:ASingularOrbit}.

\begin{figure}
\begin{center}
\begin{subfigure}[b]{0.3\textwidth}
\includegraphics[scale=.22]{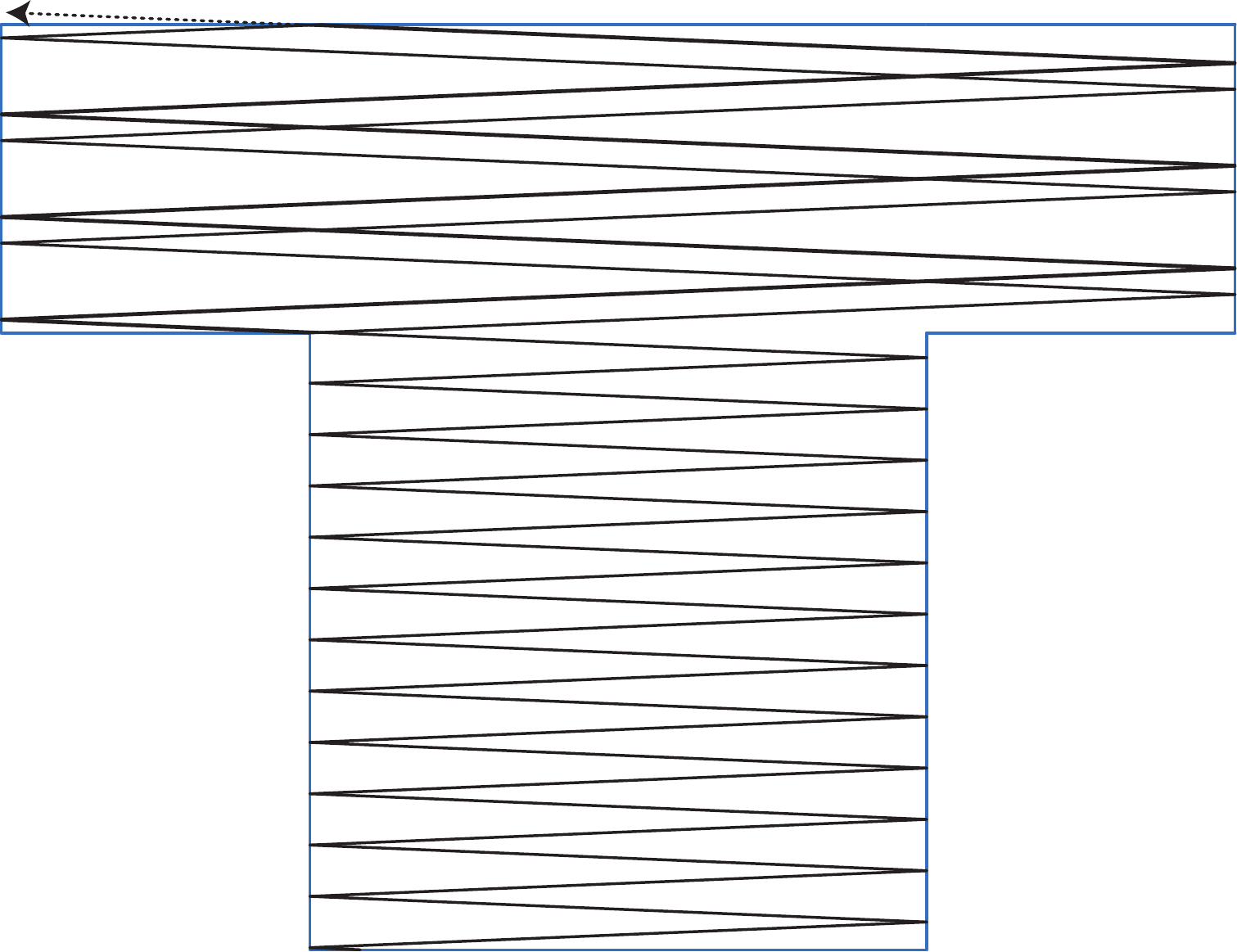}
\caption{}
\label{subfig:a}
\end{subfigure}
\begin{subfigure}[b]{0.3\textwidth}
\includegraphics[scale=.25]{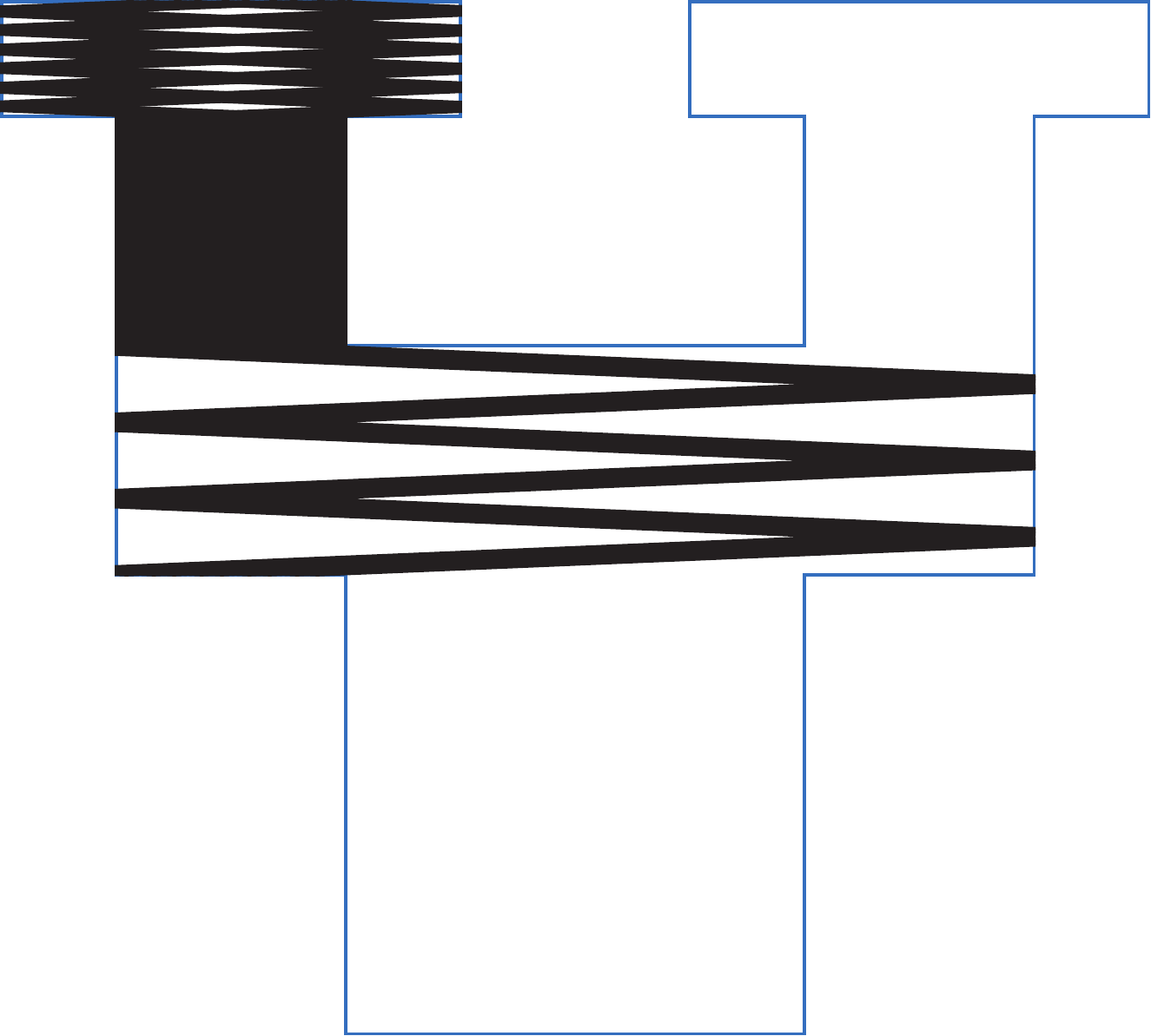}
\caption{}
\label{subfig:b}
\end{subfigure}
\begin{subfigure}[b]{0.3\textwidth}
\includegraphics[scale=.25]{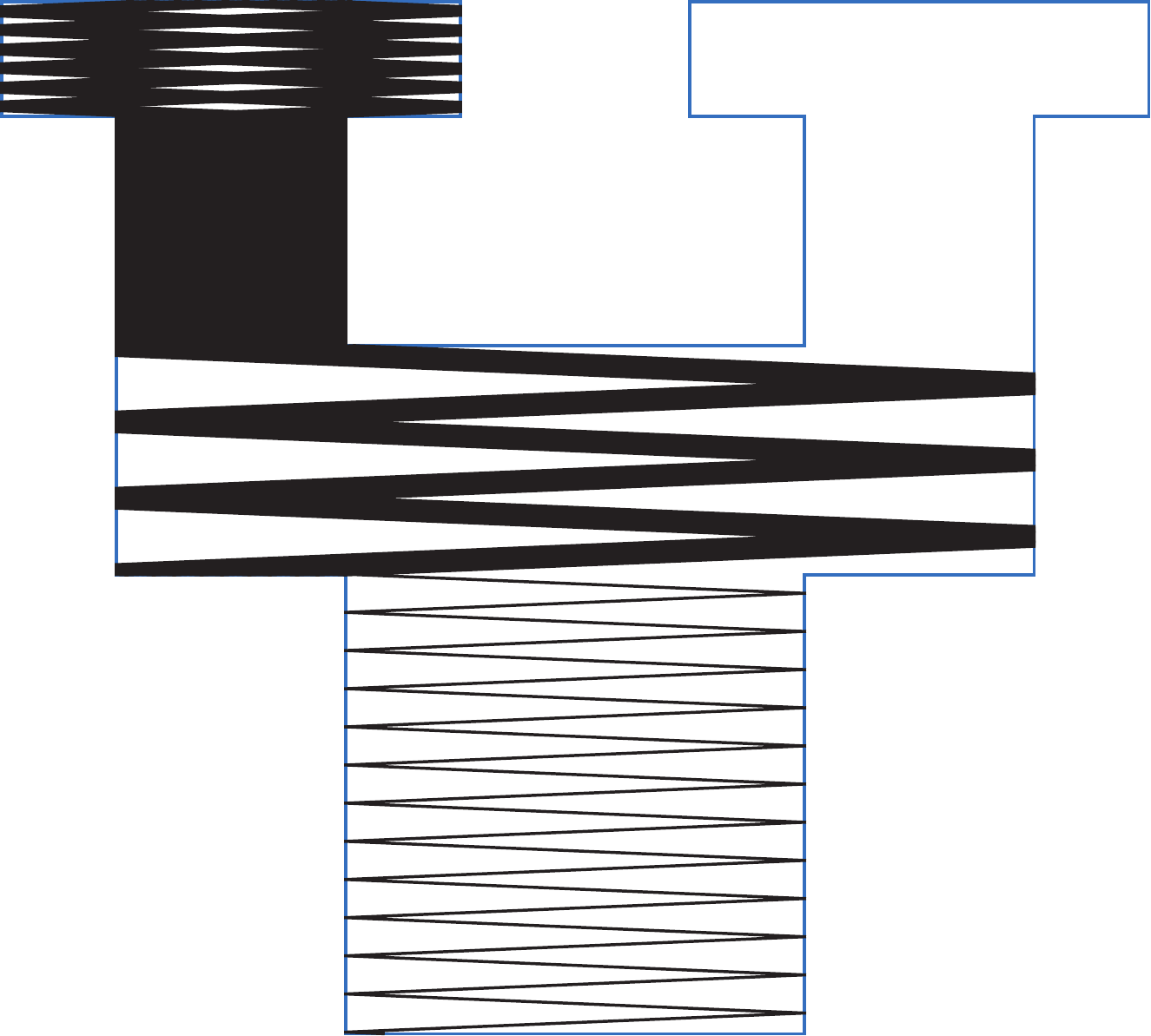}
\caption{}
\end{subfigure}
\end{center}
\caption{The figure on the left stops at the segment to be removed in the construction of $\omegati{1}$ from $\omegati{0}$. The figure in the middle precisely illustrates what happens to the billiard ball  upon entering $\omegati{1}$ from $\omegati{0}$.  The last figure is then the union of the  paths shown in Figure \ref{subfig:a} and Figure \ref{subfig:b}.  The orbit shown in the center may be scaled by $2^{-n}$, appropriately reflected and appended to $\xii{n}{\tau_n}$, $n\geq 2$ in order to produce the portion of the next orbit.}
\label{fig:explainingIrrationalOrbitConstruction}
\end{figure}

\section{Discussion}
\label{sec:Discussion}

We have begun to thoroughly understand what constitutes a periodic orbit of the $T$-fractal billiard $\omegat$.  Eventually constant sequences of periodic orbits have trivial limits constituting periodic orbits of $\omegat$.  Other, more complicated, examples involved showing that particular sequences of compatible periodic orbits yielded limiting curves that constitute periodic orbits of $\omegat$ in the form of the unions of two nontrivial paths.  More precisely, a sequence of compatible periodic orbits coming from a particular family of sequences of compatible periodic orbits could be shown to converge with respect to the Hausdorff metric to a well-defined path given by the union of two nontrivial paths derived from the sequence of compatible periodic orbits.

We wish to determine in future works whether or not every periodic orbit is either the trivial limit of an eventually constant sequence of compatible periodic orbits or the Hausdorff(--Gromov) limit of a sequence of compatible periodic orbits.  

Given our example discussed in \S\ref{sec:aNontrivialPathInAnIrrationalDirection}, we suspect that a classification of orbits on the $T$-fractal billiard will not be as straightforward as it is for square-tiled billiard tables \cite{GutJu1,GutJu2} (namely, that in a fixed direction, the billiard flow is either closed or uniquely ergodic).  Regarding the example in \S\ref{sec:aNontrivialPathInAnIrrationalDirection}, we seek to understand what happens for the billiard orbit beginning from $y^0 :=x^{\zeta_0}_0+2s_0$ in the direction of $\theta^0$ such that $\tan\theta^0 = \sqrt{2}/34$.  That is, we want to know whether or not such an orbit will also form a nontrivial path and converge to a rational elusive point.  If this is the case, then this would be an example of a sequence of compatible dense orbits yielding a nontrivial path reaching a rational elusive point.  A further question we could then ask is whether or not the union of the two nontrivial paths constitutes a periodic orbit of the $T$-fractal billiard, in the sense of Definition \ref{def:periodicOrbit}.  If this turns out to indeed be the case, then this would constitute a periodic orbit of the $T$-fractal billiard in an irrational direction, in sharp contrast with the classic results for the aforementioned square-tiled billiard tables.

\subsubsection*{Acknowledgments} The authors are very grateful to the anonymous referee for his or her  very helpful suggestions and comments, well beyond the call of duty.

\end{document}